\documentclass[reqno]{amsart}
\usepackage{amsfonts,amssymb,amsmath,latexsym,indentfirst,cite}
\usepackage{amsthm}
\usepackage{color,graphicx}
\usepackage{eucal}
\usepackage{verbatim}

\newcommand{\vertiii}[1]{{\left\vert\kern-0.25ex\left\vert\kern-0.25ex\left\vert #1 
    \right\vert\kern-0.25ex\right\vert\kern-0.25ex\right\vert}}
\newcommand{\R}{\mathbb R}

\newcommand{\N}{\mathbb N}

\newcommand{\peq}{\hspace*{0.10in}}
\newcommand{\peqq}{\hspace*{0.05in}}

\newtheorem{theorem}{Theorem}[section]
\newtheorem{proposition}[theorem]{Proposition}
\newtheorem{lemma}[theorem]{Lemma}
\newtheorem{corollary}[theorem]{Corollary}
\newtheorem{definition}[theorem]{Definition}

\theoremstyle{definition}
\newtheorem{remark}[theorem]{Remark}

\begin{document}

\title[Nonlinear Profile Decomposition]{Nonlinear Profile Decomposition and the Concentration Phenomenon for Supercritical Generalized KdV Equations}
\author[L. Farah]{Luiz G. Farah}
\address{Universidade Federal de Minas Gerais\\lgfarah@gmail.com}
\thanks{L.G.F. was partially supported by a CNPq-Brazil and FAPEMIG-Brazil.}

\author[B. Pigott]{Brian Pigott} 
\address{Wofford College\\pigottbj@wofford.edu}
\thanks{B.P. was partially supported by a Wofford College Faculty Development Grant.}



\begin{abstract}
A nonlinear profile decomposition is established for solutions of supercritical generalized Korteweg-de Vries equations. As a consequence, we obtain a concentration result for finite time blow-up solutions that are of Type II.
\end{abstract}

\maketitle

\section{Introduction}

We consider the initial value problem for the supercritical generalized Korteweg-de Vries (gKdV) equation
\begin{equation}
\label{gkdv}
\left \{ \begin{array}{l l}
\partial_{t} u + \partial_{x}^{3} u + \partial_{x} (u^{k+1}) = 0, & x \in \R, \ t>0,\\
u(x,0) = u_{0}(x),
\end{array}
\right .
\end{equation}
where $k > 4$ is an integer. In the particular case when $k =1$, this equation was derived by Korteweg and de Vries \cite{KdV} to model long waves on a shallow rectangular canal. In the present work we are primarily interested in the case when $k > 4$, the so-called \emph{$L^{2}$-supercritical} case, which is a generalization of the model derived in \cite{KdV}.

We recall that the gKdV equation \eqref{gkdv} has the following scaling symmetry: if $u(x,t)$ solves \eqref{gkdv} with initial data $u(x,t) = u_{0}(x)$, then for any $\lambda > 0$ the function $u_{\lambda}(x,t)$ given by
\begin{equation}
\label{scaling}
u_{\lambda}(x,t) = \lambda^{2/k} u(\lambda x, \lambda^{3} t)
\end{equation}
is also a solution of \eqref{gkdv} with initial data $u_{0,\lambda}(x) = u_{0}(\lambda x)$. A simple calculation reveals that
\begin{equation*}
\| u_{\lambda} \|_{\dot{H}^{s}} = \lambda^{s + \frac{2}{k} - \frac{1}{2}} \| u \|_{\dot{H}^{s}}.
\end{equation*}
We note that if $s_{k} = \frac{1}{2} - \frac{2}{k} = (k-4)/2k$, then the homogeneous Sobolev space $\dot{H}^{s_{k}}(\R)$ is invariant under the scaling symmetry. In particular, if $k = 4$, we find that the scale invariant Sobolev space is $\dot{H}^{0}(\R) = L^{2}(\R)$; we refer to the problem \eqref{gkdv} with $k=4$ as the $L^{2}$-critical (or simply critical) problem. When $k > 4$ we note that $s_{k} > 0$; these problems are referred to as $L^{2}$-supercritical (or simply supercritical). (In light of the mass and energy conservation -- see \eqref{mass} and \eqref{energy} below -- these problems are also referred to as mass-supercritical and energy-subcritical.)  Thanks to the scaling structure it is natural to study \eqref{gkdv} in the Sobolev spaces $H^{s}(\R)$ with $s \geq s_{k}$.

The local well-posedness theory for \eqref{gkdv} with $k \geq 4$ is well-understood. Kenig, Ponce, and Vega \cite{kpv1} showed that the equation is locally well-posed in the Sobolev space $H^{s}(\R)$ with $s \geq s_{k}$. Local well-posedness in the critical space $\dot{H}^{s_{k}}(\R)$ is more delicate with the length of the interval of existence depending on $u_{0} \in \dot{H}^{s_{k}}(\R)$, rather than on $\| u_{0} \|_{\dot{H}^{s_{k}}}$ only, see \cite{kpv1}. Farah and Pastor also investigated local well-posedness for \eqref{gkdv} with $k >4$ in \cite{FP13} and developed an alternative proof to the argument given in \cite{kpv1}. These results are reviewed below in Section \ref{lwpreview}, see Theorem \ref{local} and Corollary \ref{smallglobal}. The corresponding global theory in $H^1(\R)$ for \eqref{gkdv} with $k \geq 4$ is also well-understood. For the critical gKdV equation ($k=4$), a combination of the results in \cite{kpv1} and the sharp version of a Gagliardo-Nirenberg inequality proved by Weinstein \cite{W83} imply global well-posedness for intial data $u_0 \in H^1(\R)$ satisfying
\begin{equation}\label{SA}
\|u_0\|_{L^2} < \|Q\|_{L^2}, 
\end{equation}
where $Q$ is the unique positive radial solution of the elliptic equation 
\begin{equation}\label{GSE}
Q^{''} - Q + Q^{5} = 0.
\end{equation}
This result was extended for $k>4$ by Farah, Linares, and Pastor \cite{FLP} (see also Holmer and Roudenko \cite{HR08} for the corresponding result in the case of NLS equation), where they proved that the initial value problem is globally well-posed for $u_{0} \in H^{1}(\R)$ provided
\begin{equation*}
\mathcal{E}[u_{0}]^{s_{k}} \mathcal{M}[u_{0}]^{1-s_{k}} < \mathcal{E}[Q]^{s_{k}} \mathcal{M}[Q]^{1-s_{k}}, \qquad \mathcal{E}[u_{0}] \geq 0,
\end{equation*}
and
\begin{equation*}
\| \partial_{x} u_{0} \|_{L^{2}}^{s_{k}} \| u_{0} \|_{L^{2}}^{1-s_{k}} < \| \partial_{x} Q \|_{L^{2}}^{s_{k}} \| Q \|_{L^{2}}^{1-s_{k}},
\end{equation*}
where $\mathcal{M}$ and $\mathcal{E}$ refer to the mass and energy quantities conserved by the gKdV equation \eqref{gkdv} (see definitions \eqref{mass} and \eqref{energy} below). Also $Q$ stands for the unique positive radial solution of the elliptic equation 
\begin{equation*}
Q^{''} - Q + Q^{k+1} = 0.
\end{equation*}
It should be pointed out that in the defocusing case (replace the `$+$' sign in front of the nonlinear term with a `$-$' sign), the $I$-method can be used to prove that the initial value problem is globally well-posed for $u_{0} \in H^{s}(\R)$ for $s > 4(k-1)/5k$ without any smallness condition on the initial data; see \cite{FLP}. In the focusing case, the $I$-method can also be applied to the critical gKdV equation ($k=4$) under the smallness assumption \eqref{SA}, see for instance \cite{FLP1}, \cite{LG6} and \cite{miao}. In this paper we consider only the focusing case for \eqref{gkdv}.

The following two quantities are conserved by \eqref{gkdv}:
\begin{align}
\text{(Mass)} \quad \mathcal{M}[u(t)] &:= \int_{\R} \vert u(x,t) \vert^{2} dx = \int_{\R} \vert u_{0}(x) \vert^{2} dx = \mathcal{M}[u_{0}]  \label{mass} \\
\text{(Energy)} \quad \mathcal{E}[u(t)] &:= \int_{\R} \left ( u_{x}^{2}(x,t) - \frac{2}{k+2} u^{k+2}(x,t) \right ) dx \label{energy} \\
&= \int_{\R} \left ( u_{0x}^{2}(x) - \frac{2}{k+2} u_{0}^{k+2}(x) \right ) dx = \mathcal{E}[u_{0}] . \notag
\end{align}
In the critical case ($k=4$) we notice that the mass is invariant under the scaling symmetry and is a conserved quantity, while in the supercritical case ($k > 4$), the scale invariant Sobolev space no longer coincides with a conserved quantity. This presents a key difficulty in the analysis to come. 

Mass-supercritical problems have received a great deal of attention in recent years. A new strategy, the so-called concentration-compactness/rigidity method, developed by Kenig and Merle in \cite{KM2006} and \cite{KM2008} allowed the authors to prove sharp results on energy-critical problems for the nonlinear Schr\"{o}dinger equation and the nonlinear wave equation. 
Later these ideas were adapted to address mass-supercritical, energy-subcritical problems, see \cite{KM2010}, \cite{HR08}, \cite{DHR2008}, \cite{FXC11} and \cite{Guevara2014}, for instance. A detailed account of the ideas employed in these results is available in \cite{Kenig2015}.

A key ingredient in the concentration-compactness/rigidity method is the profile decomposition. The idea of this decomposition is connected to the concentration compactness method of P.L. Lions and the bubble decomposition for elliptic equations. For dispersive equations the profile decomposition was first introduced by Merle and Vega \cite{MV1998} and Bahouri and Gerard \cite{BaGe99} in the context of Schr\"odinger and wave equation, respectively.
The principle goal of this article is to establish a nonlinear profile decomposition result for the supercritical KdV equations.  We begin by recalling the following linear profile decomposition result proved by Farah and Versieux \cite{FV15}. (See also \cite{shaolpd} for related results when $k=4$.)

\begin{theorem} \label{Lprofdec}
 Let $k>4$, $s_k= (k-4)/2k$ and $\{\phi_n\}_{n\in \mathbb{N}}$ be a bounded sequence in $\dot{H}^{s_k}(\R)$. Set $v_n=V(t)\phi_n$, where $V(t)$ is the Airy evolution. Then there exists a subsequence, which we still denote by $\{v_n\}_{n\in \mathbb{N}}$, a sequence of functions $\{\psi^j\}_{j\in \mathbb{N}}\subset \dot{H}^{s_k}(\R)$ and sequences of parameters $(h_n^j, x_n^j, t_n^j)_{n\in \mathbb{N}, j\in \mathbb{N}}$  such that for every $J\geq1$ there exists $R^J_n\in \dot{H}^{s_k}(\R)$ satisfying
\begin{equation}\label{SUM}
v_n(x,t)=\sum_{j=1}^{J}\dfrac{1}{(h_n^j)^{2/k}}V\left(\dfrac{t-t_n^j}{(h_n^j)^{3}}\right)\psi^j\left(\dfrac{x-x_n^j}{h_n^j}\right) + V(t)R_n^J
\end{equation}
where the remainder sequence has the following asymptotic smallness property
\begin{equation}\label{LSPWNL}
\limsup_{n\rightarrow \infty} \|D_x^{1/p}V(t)R_n^J\|_{L^p_tL^{q}_x} \rightarrow 0,  \peqq \textrm{as} \peqq J\rightarrow \infty,
\end{equation}
for all $(p,q)$ satisfying the condition
\begin{equation*}
\frac{2}{p} + \frac{1}{q} = \frac{2}{k}.
\end{equation*}

Moreover the remainder also asymptotically vanishes in the Strichartz space $L_x^{5k/4}L^{5k/2}_t$, meaning that
\begin{equation}\label{LSPWNL2}
\limsup_{n\rightarrow \infty} \|V(t)R_n^J\|_{L_x^{5k/4}L^{5k/2}_t} \rightarrow 0,  \peqq \textrm{as} \peqq J\rightarrow \infty.
\end{equation}

Furthermore, the sequences of parameters have a pairwise divergence property: For all $1\leq i \neq j\leq J$, 
\begin{equation}\label{LXT}
\lim_{n\rightarrow \infty} \left|\dfrac{h_n^i}{h_n^j}\right|+\left|\dfrac{h_n^j}{h_n^i}\right|+\left|\dfrac{t_n^i-t_n^j}{(h_n^i)^{3}}\right|+\left|\dfrac{x_n^i-x_n^j}{h_n^i}\right|=\infty.
\end{equation}

Finally, for fixed $J\geq 1$, we have the asymptotic Pythagorean expansion
\begin{equation}\label{LPYTHA}
\|v_n(\cdot,0)\|^2_{\dot{H}^{s_k}(\R)}-\sum_{j=1}^{J}\|\psi^j\|^2_{\dot{H}^{s_k}(\R)} - \|R_n^J\|^2_{\dot{H}^{s_k}(\R)} \rightarrow 0,  \peqq \textrm{as} \peqq n\rightarrow \infty.
\end{equation}
\end{theorem}

Before stating our main result we require the following definition.

\begin{definition} \label{NLP}
Let $\psi \in \dot{H}^{s_{k}}(\R)$ and $\{ t_{n} \}_{n \in \N}$ a sequence with $\lim_{n \to \infty} t_{n} = \overline{t} \in [-\infty, \infty]$. We say that $u(x,t)$ is a nonlinear profile associated with $(\psi, \{ t_{n} \}_{n \in \N})$ if there exists an interval $I = (a,b)$ with $\overline{t} \in I$ (if $\overline{t} = \pm \infty$, then $I = (a, +\infty)$ or $I = (-\infty, b)$, as appropriate) such that $u$ solves \eqref{gkdv} in $I$ and 
\begin{equation*}
\lim_{n \to \infty} \| u(\cdot, t_{n}) - V(t_{n}) \psi \|_{\dot{H}^{s_{k}}} = 0.
\end{equation*}
\end{definition}

The existence and uniqueness of nonlinear profiles is demonstrated in Section \ref{nlpd} (see Proposition \ref{existnlpd} and Remark \ref{4.2}). Our main result reads as follows.

\begin{theorem}[Nonlinear Profile Decomposition]
\label{Nprofdec}
 Assume $k>4$,  $s_k= (k-4)/2k$ and $\{\phi_n\}_{n\in \mathbb{N}}$ be a bounded sequence in $\dot{H}^{s_k}(\R)$. Let $\{\psi^j\}_{j\in \mathbb{N}}\subset \dot{H}^{s_k}(\R)$ and $(h_n^j, x_n^j, t_n^j)_{n\in \mathbb{N}, j\in \mathbb{N}}$ be, respectively, the sequence of functions and sequences of parameters given by Theorem \ref{Lprofdec} and $\{U^j\}_{j\in \mathbb{N}}$ the family of nonlinear profiles associated with $(\psi^j, \{-t^j_n/(h_n^j)^3\}_{n\in \N})_{j\in \N}$.
 
 Let $\{I_n\}_{n\in \mathbb{N}}$  be a sequence of intervals containing zero. The following statements are equivalents
 \begin{itemize}
 \item [$(i)$] 
 \begin{equation}\label{(i)}
 \lim_{n\rightarrow \infty} \left(\|D_x^{s_k}U_n^j\|_{L^{5}_xL^{10}_{I_n}}+\|U_n^j\|_{L_x^{5k/4}L^{5k/2}_{I_n}}\right) <\infty \quad \textrm{for every} \quad j\geq 1.
 \end{equation} 
 \item [$(ii)$] 
 \begin{equation}\label{(ii)}
 \lim_{n\rightarrow \infty}  \left(\|D_x^{s_k}u_n\|_{L^{5}_xL^{10}_{I_n}} +\|u_n\|_{L_x^{5k/4}L^{5k/2}_{I_n}}\right)< \infty,
 \end{equation} 
 \end{itemize}
 where 
 \begin{equation}\label{Unj}
 U_n^j(x,t)=\dfrac{1}{(h_n^j)^{2/k}}U^j\left(\dfrac{x-x_n^j}{h_n^j}, \dfrac{t-t_n^j}{(h_n^j)^{3}}\right)
 \end{equation}
 and
 $u_n$ is the solution to the gKdV equation \eqref{gkdv} with Cauchy data $\phi_n$ at $t=0$.
 
 Moreover, if \eqref{(i)} or \eqref{(ii)} holds, then (up to a subsequence) for every $J\geq1$
 \begin{equation}\label{Decomp}
u_n=\sum_{j=1}^{J}U_n^j + V(t)R_n^J + r_n^J,
\end{equation}
 with
 \begin{equation}\label{NSPWNL}
\limsup_{n\rightarrow \infty} \left(\|r_n^J\|_{L^{\infty}_{I_n}\dot{H}_x^{s_k}}+\|D_x^{s_k}r_n^J\|_{L^{5}_xL^{10}_{I_n}}+\|r_n^J\|_{L_x^{5k/4}L^{5k/2}_{I_n}}\right) \rightarrow 0,  \peqq \textrm{as} \peqq J\rightarrow \infty,
\end{equation}
where $\{R^J_n\}_{n,J\in\mathbb{N}}\subset \dot{H}^{s_k}(\R)$ are as Theorem \ref{Lprofdec}.
\end{theorem}

As an application of the nonlinear profile decomposition Theorem \ref{Nprofdec} we now prove a concentration result for blow-up solutions of \eqref{gkdv} with $k > 4$.

In the case of the critical problem $k=4$, it is known that the maximum time of existence may be finite \cite{MM02}. More precisely, there exists $u_{0} \in H^{1}(\R)$ such that the corresponding solution $u(x,t)$ of \eqref{gkdv} with $k =4$ blows up at finite time $T^{\ast}$:
\begin{equation*}
\lim_{t \uparrow T^{\ast}} \| u(t) \|_{H^{1}} = \infty.
\end{equation*}
Recent work of Martel, Merle, and Rapha\"{e}l \cite{MMR1}, \cite{MMR2}, \cite{MMR3} offers an updated perspective on these results .
Although it is expected that the supercritical problems $k > 4$ also admit finite time blow up solutions, the problem remains open. There are a number of numerical results that suggest that this is the case, see \cite{DM1998}, for instance. In recent work of Koch \cite{koch2015} and Lan \cite{lan2015}, blow-up solutions to \eqref{gkdv} with a slightly supercritical nonlinearity have been constructed. These results do not fall within the class of nonlinearities we consider here. We do not address the existence of finite time blow up solutions in the present work. Rather we demonstrate the concentration phenomenon for the supercritical generalized KdV equations assuming the existence of finite time blow-up solutions. Mass concentration results for the critical generalized KdV equation are known, see \cite{KPV2000} and \cite{pigottconc}. Theorem \ref{Concentration} below extends these critical concentration results to the supercritical gKdV equation \eqref{gkdv} in the critical Sobolev space $\dot{H}^{s_{k}}(\R)$. Our main assumption is that the blow up solution is of type II, that is, the solution blows up and remains bounded in the critical Sobolev norm. In our case,
\begin{equation}
\label{TypeII1}
\sup_{t \in [0,T^{\ast})} \| u(t) \|_{\dot{H}^{s_{k}}} < \infty,
\end{equation}
where $T^{\ast} > 0$ is the blow up time. It should be pointed out that the local well–posedness theory does not rule out type II solutions.

An abundance of recent literature is devoted to the study type II blow-up solutions for several dispersive models. For instance, in the case of energy-critical wave equation Krieger, Schlag and Tataru \cite{KST09}, Krieger and Schlag \cite{KS14}, Hillairet and Raph\"ael \cite{HR12} and Jendrej \cite{J} constructed examples of this type of solutions. Moreover, the works of Duyckaerts, Kenig and Merle \cite{DKM11}-\cite{DKM12} characterize these solutions. For the energy supercritical NLS, the first example of type II blow-up solutions is due Merle, Rapha\"el and Rodnianski \cite{MRRPr}.

A concentration result for supercritical nonlinear Schr\"{o}dinger (NLS) equations has recently been established by Guo \cite{Guo2013}. There the author considers the following initial value problem
\begin{equation}
\label{nls}
\left \{ \begin{array}{l l}
i u_{t} + \Delta u + \vert u \vert^{p-1} u = 0, & x \in \R^{d}, t \geq 0\\
u(x,0) = u_{0}(x) 
\end{array}
\right .
\end{equation}
with $p > 1+ 4/d$
and establishes a concentration result, provided the blow up is of Type II, meaning that
\begin{equation}
\label{TypeII}
\sup_{t \in [0,T^{\ast})} \| u(t) \|_{\dot{H}^{s_{p}}} < \infty,
\end{equation}
where $s_{p} = \frac{d}{2} - \frac{2}{p-1} $ denotes the critical Sobolev regularity. 
Guo shows that if $u$ is a solution of \eqref{nls} that blows up at time $T^{*} >0$ and satisfies \eqref{TypeII}, then 
\begin{equation*}
\liminf_{t \uparrow T^{*}} \int_{\vert x - x(t) \vert \leq \lambda(t)} \left \vert (-\Delta)^{s_{p}/2} u(x,t) \right \vert^{2} dx \geq \| \widetilde{Q} \|_{\dot{H}^{s_{p}}}^{2},
\end{equation*}
where $\lambda(t) > 0$ satisfies $\lambda(t) \| \nabla u(t) \|_{L^{2}}^{1/(1-s_{p})} \to \infty$ as $t \uparrow T^{*}$, $\widetilde{Q}$ solves
\begin{equation*}
-\Delta Q + \frac{p-1}{2} (-\Delta)^{s_{p}}Q - \vert Q \vert^{p-1} Q = 0,
\end{equation*}

It should be noted that the work of Merle and Raphael \cite{MR2008} shows that it is possible, in the case of supercritical NLS, for the solution to violate the type II blow up condition (assuming $d\geq 3$ and radial initial data). Indeed, they show that there are solutions $u(x,t)$ of \eqref{nls} that blow up at finite time $T^{\ast}$ such that
\begin{equation*}
\lim_{t \uparrow T^{\ast}} \| u(t) \|_{\dot{H}^{s_{p}}} = \infty.
\end{equation*}
Currently no such result exists for supercritical KdV equations \eqref{gkdv}.

Before state our concentration result, we recall that the small data global theory (see Corollary \ref{smallglobal}) implies the existence of a number $\delta_k>0$ such that for any $u_{0} \in \dot{H}^{s_{k}}(\R)$ with $\| u_{0} \|_{\dot{H}^{s_{k}}} < \delta_{k}$, the corresponding solution $u$ of \eqref{gkdv} with $u(0) = u_{0}$ is global in time. Moreover, the solution also satisfies
\begin{equation}
\label{a2}
u\in L^{\infty}_{t} \dot{H}^{s_{k}} \cap L^{5k/4}_{x} L^{5k/2}_{t} \quad \textrm{and} \quad D^{s_{k}}_{x} u \in L^{5}_{x} L^{10}_{t}.
\end{equation}

In light of the above result we can introduce the following definition.

\begin{definition} 
\label{DefDelta}
We define $\delta_{0} = \delta_{0}(k)$ as the supremum over $\delta_{k}$ such that global existence for \eqref{gkdv} holds and the solution $u$ satisfies \eqref{a2}.
\end{definition}

We are now in position to state our main concentration theorem. It asserts that for every type II blow-up solution for the gKdV equation \eqref{gkdv} such that $\delta_{0} \leq \| u(t) \|_{\dot{H}^{s_{k}}} \leq (3\sqrt{2}/4) \delta_{0}$ there is a concentration phenomenon in the $\dot{H}^{s_{k}}(\R)$ norm at the blow up time, with minimal amount $\delta_0$.

\begin{theorem}\label{Concentration}
Let $\delta_{0}>0$ given in Definition \ref{DefDelta} and $u\in C([0,T^{\ast}):\dot{H}^{s_k}(\R))$ be a  solution of the gKdV equation \eqref{gkdv} with $k > 4$ which blows up at finite time $T^{\ast}<\infty$ such that $u(t)\in C = \{ f \in \dot{H}^{s_{k}}(\R) \  \vert \ \delta_{0} \leq \| f \|_{\dot{H}^{s_{k}}} \leq (3\sqrt{2}/4) \delta_{0} \}$ for all $t\in [0,T^{\ast})$. Let $\lambda(t)>0$ such that $\lambda(t)^{-1} (T^{\ast}-t)^{1/3}\rightarrow 0$ as $t\rightarrow T^{\ast}$. There exists $x(t)\in \R$ such that
\begin{equation}\label{Concentration2}
\liminf_{t\rightarrow T^{\ast}}\int_{|x-x(t)|\leq \lambda(t)}|D^{s_k}_xu(x,t)|^2dx\geq \delta_0^2.
\end{equation} 
\end{theorem}


\begin{remark}
The assumption $\|u(t)\|_{\dot{H}^{s_k}}\leq (3\sqrt{2}/4) \delta_0$ is technical and it guarantees the uniqueness of the blowing up profiles given by Theorem \ref{Nprofdec} below. Without this assumption we cannot prove Theorem \ref{Concentration}.
\end{remark}

The proof of Theorem \ref{Concentration} is inspired by the work of Keraani \cite{Keraani2006} where the author establishes a concentration result for the critical nonlinear Schr\"{o}dinger equation 
\eqref{nls}
with $d = 1,2$ and $p = 1 + 4/d$ by first establishing a nonlinear profile decomposition result for the solutions.

\subsection{Organization} 
The paper is organized as follows. In Section \ref{notation} we review the notation that is to be used throughout the paper. Section \ref{lwpreview} offers a review of the Strichartz estimates and the local well-posedness theory for the supercritical generalized KdV equations including criteria for blow-up. The proof of Theorem \ref{Nprofdec} is contained in Section \ref{nlpd}, and the proof of Theorem \ref{Concentration} is given in Section \ref{conc}.

\section{Notation} \label{notation}

In this section we introduce the notation that will be used throughout the paper. We use $c$ to denote various constants that may vary from one line to the next. Given any positive quantities $a$ and $b$, the notation $a \lesssim b$ means that $a \leq cb$, with $c$ uniform with respect to the set where $a$ and $b$ vary. Also, we denote $a \sim b$ when, $a \lesssim b$ and $b \lesssim a$. 

We write $\| f \|_{L^{p}}$ for the norm of $f$ in $L^{p}(\R)$. We also use the mixed norm space $L^{q}_{t} L^{r}_{x}$ to denote the space of space-time functions $u(x,t)$ for which the norm 
\begin{equation*}
\| u \|_{L^{q}_{t} L^{p}_{x}} := \left ( \int_{-\infty}^{\infty} \| u(t) \|_{L^{r}_{x}}^{q} dt \right )^{1/q}
\end{equation*}
is finite, with the usual modifications if either $q = \infty$ or $r = \infty$. In some cases we wish to consider a finite time interval $I = [a,b]$, in which case we write
\begin{equation*}
\| u \|_{L^{q}_{I} L^{p}_{x}} = \| u \|_{L^{q}_{[a,b]} L^{p}_{x}} = \left ( \int_{a}^{b} \| u(t) \|_{L^{p}_{x}}^{q} dt \right )^{1/q}.
\end{equation*}
The spaces $L^{p}_{x} L^{q}_{t}$ and $L^{p}_{x} L^{q}_{I} = L^{p}_{x} L^{q}_{[a,b]}$ are defined similarly.

We define the spatial Fourier transform of a function $f(x)$ by
\begin{equation*}
\widehat{f}(\xi) = \int_{-\infty}^{\infty} e^{-ix \xi} f(x) dx.
\end{equation*}

The class of Schwartz functions is denoted by $\mathcal{S}(\R)$. We define $D^{s}_{x}$ to be the Fourier multiplication operator with symbol $\vert \xi \vert^{s}$. In this case the (homogeneous) Sobolev space $\dot{H}^{s}(\R)$ is collection of functions $f:\R \to \R$ equipped with the norm
\begin{equation*}
\| f \|_{\dot{H}^{s}} = \| D^{s}_{x} f \|_{L^{2}}.
\end{equation*}

\section{Review of the Local Well-Posedness Theory} \label{lwpreview}

In this section we shall recall the well-posedness theory for the supercritical gKdV equations, \eqref{gkdv} with $k > 4$ in the critical Sobolev space $\dot{H}^{s_k}(\R)$. We begin by recalling the Strichartz estimates associated with the Airy evolution
\begin{equation}
\label{Lkdv}
\left \{ \begin{array}{l l}
\partial_{t} u + \partial_{x}^{3}u = 0, & x \in \R, \  t>0,\\
u(0,x) = u_{0}(x)
\end{array}
\right .
\end{equation}
The solution of \eqref{Lkdv} is given by $u(x,t)=V(t)u_0(x)$, where $V(t):= \exp(-it\partial_{x}^{3})$ is the linear propagator for the Airy equation. Notice that the solution is globally defined in the Sobolev space $\dot{H}^s(\R)$, for all $s\in \R$. Moreover, $\{V(t)\}_{t\in \R}$ defines a unitary operator in these spaces. In particular, we have for all $s\in \R$
\begin{equation}\label{HSP}
\|V(t)u_0\|_{\dot{H}^s}=\|u_0\|_{\dot{H}^s}, \peq \textrm{ for all } t\in \R.
\end{equation}

Next, we recall some Strichartz type estimates associated to the linear propagator.

\begin{lemma}\label{lemma1} Let $k\geq 4$ and $s_k= (k-4)/2k$. Then we have the following estimates:
\begin{itemize}
\item [$(i)$] $\|V(t)u_0\|_{L^{5k/4}_xL^{5k/2}_t}\leq c\|u_0\|_{\dot{H}^{s_k}}$;
\item [$(ii)$] $\|V(t)u_0\|_{L^{5}_xL^{10}_t}\leq c\|u_0\|_{L^2}$;
\item [$(iii)$] $\|\partial_x\int_0^tV(t-s)g(\cdot, s)ds\|_{L^{\infty}_tL^2_x}+ \|\int_0^tV(t-s)g(\cdot, s)ds\|_{L^{5}_xL^{10}_t} \leq c \|g\|_{L^{1}_xL^{2}_t}$;
\item [$(iv)$]$\|\partial_x\int_0^tV(t-s)g(\cdot, s)ds\|_{L^{5k/4}_xL^{5k/2}_t}\leq c \|D_x^{s_k}g\|_{L^{1}_xL^{2}_t}.$
\end{itemize}
\end{lemma}

\begin{proof}
Inequalities $(i)$ and $(ii)$ were proved, respectively, by Farah and Pastor \cite{FP13} (Lemma 2.5) and Kenig, Ponce and Vega \cite{kpv1} (Corollary 3.8). The inequalities $(iii)$ and $(iv)$ follows from $(i)$, $(ii)$ and \eqref{HSP} by way of duality and a $TT^{\ast}$ argument.
\end{proof}

Further well-known Strichartz estimates are the following

\begin{lemma}\label{lemma12} Let $k\geq 4$ and $s_k= (k-4)/2k$. Then we have the following estimates:
\begin{itemize}
\item [$(i)$] $\|D^{1+s_k}_x V(t)u_0\|_{L^{\infty}_xL^{2}_t}+\|D^{-1/k}_x V(t)u_0\|_{L^{k}_xL^{\infty}_t}\leq c\|u_0\|_{\dot{H}^{s_k}}$;
\item [$(ii)$] $\|V(t)u_0\|_{L_x^{\frac{k(3k-2)}{3k-4}}L^{3k-2}_{t}}\leq c\|u_0\|_{\dot{H}^{s_k}}$;
\item [$(iii)$] $\|D^{1+s_k}_x\partial_x\int_0^tV(t-s)g(\cdot, s)ds\|_{L^{\infty}_tL^2_x}\leq c \|D_x^{s_k}g\|_{L^{1}_xL^{2}_t}.$
\item [$(iv)$] $\|\partial_x\int_0^tV(t-s)g(\cdot, s)ds\|_{L_x^{\frac{k(3k-2)}{3k-4}}L^{3k-2}_{t}} \leq c \|D_x^{s_k}g\|_{L^{1}_xL^{2}_t}.$
\end{itemize}
\end{lemma}

\begin{proof}
For $(i)$ see Kenig, Ponce and Vega \cite{kpv1} (Theorem 3.5 and Lemma 3.29). The inequality $(ii)$ is a interpolation between the two inequalities in $(i)$. Finally, $(iii)$ can be obtained by duality and a $TT^{\ast}$ argument combined with $(i)$ and $(ii)$.
\end{proof}

\begin{remark} 
Note that when $k=4$ the estimates in Lemma \ref{lemma1} $(i)$-$(ii)$ and in Lemma \ref{lemma12} $(ii)$ are all the same.
\end{remark}

We say that a pair $(p,q)$ is $\dot{H}^{s_{k}}$-admissible if 
\begin{equation}
\label{Adpair}
\frac{2}{p} + \frac{1}{q} = \frac{2}{k}.
\end{equation}
In this case we have the following Strichartz type estimate.

\begin{lemma}\label{lemma2} If $k\geq 4$, $s_k= (k-4)/2k$, and $(p,q)$ is an $\dot{H}^{s_{k}}$-admissible pair,  then 
\begin{equation}\label{STR0}
 \|D^{1/p}_xV(t)u_0\|_{L_t^pL_x^q}\leq c\|u_0\|_{\dot{H}^{s_k}}.
\end{equation}
\end{lemma}

\begin{proof}
See Kenig, Ponce and Vega\cite[Theorem 2.1]{KPV4}.
\end{proof}

The particular case when $p=q$ in the above estimate will be useful in the sequel. In this case we find that
\begin{equation}
\label{STR}
\| D^{2/3k}_{x} V(t)u_{0} \|_{L^{3k/2}_{t,x}} \leq c \| u_{0} \|_{\dot{H}^{s_{k}}}.
\end{equation}

Also, recall the fractional Leibniz rule established by Kenig, Ponce, and Vega in \cite{kpv1} (see Theorems A.6, A.8, and A.13).

\begin{lemma} 
\label{Leibniz}
Let $0 < \alpha < 1$ and $p, p_{1}, p_{2}, q, q_{1}, q_{2} \in (1, \infty)$ with 
\begin{equation*}
\frac{1}{p} = \frac{1}{p_{1}} + \frac{1}{p_{2}} \qquad \text{and} \qquad \frac{1}{q} = \frac{1}{q_{1}} + \frac{1}{q_{2}}.
\end{equation*}
Then
\begin{itemize}
\item[$(i)$] $\displaystyle \| D^{\alpha}_{x}(fg) - f D^{\alpha}_{x} g - g D^{\alpha}_{x} f \|_{L^{p}_{x} L^{q}_{t}} \lesssim \| D^{\alpha}_{x}f \|_{L^{p_{1}}_{x} L^{q_{1}}_{t}} \| g \|_{L^{p_{2}}_{x} L^{q_{2}}_{t}}$. \vspace{0.2cm}\\
This result holds in the case $p=1, q=2$ as well. 

\vspace{0.2cm}

\item[$(ii)$] $\displaystyle \| D^{\alpha}_{x} F(f) \|_{L^{p}_{x} L^{q}_{t}} \lesssim \| D^{\alpha}_{x}f \|_{L^{p_{1}}_{x} L^{q_{2}}_{t}} \| F'(f) \|_{L^{p_{2}}_{x} L^{q_{2}}_{t}},$ where $F\in C^1(\mathbb{R})$.
\end{itemize}
\end{lemma}

The following theorem is the local theory as proved by Farah and Pastor \cite{FP13}. We include a summary of the proof for the reader's convenience and since some of the estimates developed in the proof will be used in the ensuing analysis.

\begin{theorem}[Small Data Local Theory]
\label{local}
Let $k \geq 4,\ s_{k} = (k-4)/2k,\ u_{0} \in \dot{H}^{s_{k}}(\R)$ with $\| u_{0} \|_{\dot{H}^{s_{k}}} \leq K$, and $t_{0} \in I$, a time interval. There exists $\delta = \delta(K) > 0$ such that if
\begin{equation*}
\| V(t-t_{0}) u_{0} \|_{L^{5k/4}_{x} L^{5k/2}_{t}} < \delta,
\end{equation*}
there exists a unique solution $u$ of the integral equation
\begin{equation}\label{IntEq}
u(t) = V(t-t_{0}) u_{0} - \int_{t_{0}}^{t} V(t -t') \partial_{x} (u^{k+1})(t') dt'
\end{equation}
in $I \times \R$ with $u \in C(I; \dot{H}^{s_{k}}(\R))$ satisfying
\begin{equation}
\label{i0}
\| u \|_{L^{5k/4}_{x} L^{5k/2}_{I}} \leq 2\delta \qquad \text{and} \qquad \| u \|_{L^{\infty}_{I} \dot{H}^{s_{k}}_{x}} + \| D^{s_{k}}_{x} u \|_{L^{5}_{x} L^{10}_{I}} < 2cK,
\end{equation}
for some positive constant $c$.
\end{theorem}

\begin{proof}
We define
$$
X^k_{a,b}=\{ u \textrm{ on } I\times \R:
\|u\|_{L^{5k/4}_xL^{5k/2}_I}\leq a  \quad \textrm{and} \quad \|D^{s_k}_xu\|_{L^5_xL^{10}_I}\leq b\}.
$$
Let
\begin{equation}\label{Phi}
\Phi(u)(t):=V(t-t_0)u_0-\int_{t_0}^tV(t-t')\partial_x(u^{k+1})(t')dt',
\end{equation}
an integral operator defined on $X^k_{a,b}$. We will next choose $a$, $b$ and $\delta$ such that $\Phi:X^k_{a,b} \rightarrow X^k_{a,b}$ is a contraction.

First note that
\begin{equation*}
\begin{split}
\|\Phi(u)\|_{L^{5k/4}_xL^{5k/2}_I}&\leq
\|V(t-t_0)u_0\|_{L^{5k/4}_xL^{5k/2}_I}+c\|D^{s_k}_x(u^{k+1})\|_{L^1_xL^2_I},
\end{split}
\end{equation*}
where we have used Lemma \ref{lemma1} $(iv)$.

On the other hand
\begin{equation*}
\begin{split}
\|D_x^{s_k}\Phi(u)\|_{L^{5}_xL^{10}_I}&\leq cK+c\|D^{s_k}_x(u^{k+1})\|_{L^1_xL^2_I},
\end{split}
\end{equation*}
by Lemma \ref{lemma1} $(i)$-$(iii)$.

Using the ideas employed in Kenig, Ponce, Vega \cite{kpv1} equation (6.1) (see also Farah, Pastor \cite{FP13} equation (3.32)) we obtain
\begin{equation}\label{LocEst}
\begin{split}
\|D^{s_k}_x(u^{k+1})&\|_{L^1_xL^2_I}\leq c \|u\|^k_{L^{5k/4}_xL^{5k/2}_I}\|D_x^{s_k}u\|_{L^{5}_xL^{10}_I}\leq c a^kb
\end{split}
\end{equation}
Therefore, choosing $b=2cK$ and $a$ such that $ca^k\leq1/2$, we have 
\begin{equation*}
\begin{split}
\|D_x^{s_k}\Phi(u)\|_{L^{5}_xL^{10}_I}&\leq cK +ca^kb\leq b.
\end{split}
\end{equation*}
Now, choosing $\delta=a/2$ and $a$ so that $ca^{k-1}b\leq1/2$ yields
$$
\|\Phi(u)\|_{L^{5k/4}_xL^{5k/2}_I}\leq a.
$$
So that $\Phi:X^k_{a,b} \rightarrow X^k_{a,b}$ is well defined. 

Next, for the contraction, we first observe that Lemma \ref{lemma1} yields
\begin{equation}\label{Cont}
\begin{split}
\vertiii{\Phi(u)-\Phi(v)} &\leq c\|D^{s_k}_x(u^{k+1}-v^{k+1})\|_{L^1_xL^2_I},
\end{split}
\end{equation}
where we have set
$$
\vertiii{u}=\|u\|_{L^{5k/4}_xL^{5k/2}_I}+\|D^{s_k}_xu\|_{L^{5}_xL^{10}_I}.
$$

Since $u^{k+1}-v^{k+1}=(u-v)(\sum_{j=0}^k u^{k-j}v^j)$, we use the Leibniz rule for fractional derivatives (Lemma \ref{Leibniz}) and Holder's inequality to bound the right hand side of \eqref{Cont} by 
\begin{align*}
&\|D^{s_k}_x(u^{k+1}-v^{k+1})\|_{L^1_xL^2_I} \\ \leq c &\|\sum_{j=0}^k u^{k-j}v^j\|_{L^{5/4}_xL^{5/2}_I}\|D^{s_k}_x(u-v)\|_{L^{5}_xL^{10}_I}\\
&+ c \|(u-v)D^{s_k}_x(\sum_{j=0}^k u^{k-j}v^j)\|_{L^1_xL^2_I}\\
\leq c & (\|u\|^k_{L^{5k/4}_xL^{5k/2}_I}+\|v\|^k_{L^{5k/4}_xL^{5k/2}_I})\|D^{s_k}_x(u-v)\|_{L^{5}_xL^{10}_I}\\
&+ c \|u-v\|^k_{L^{5k/4}_xL^{5k/2}_I}(\|D^{s_k}_xu\|_{L^{5}_xL^{10}_I}\|u\|^{k-1}_{L^{5k/4}_xL^{5k/2}_I}\\
&+ \|D^{s_k}_xv\|_{L^{5}_xL^{10}_I}\|v\|^{k-1}_{L^{5k/4}_xL^{5k/2}_I}+\sum_{j=1}^{k-1}\|D^{s_k}_x(u^{k-j}v^j)\|_{L^{p_0}_xL^{q_0}_I}),
\end{align*}
where $\dfrac{1}{p_0}=1-\dfrac{4}{5k}$ and $\dfrac{1}{q_0}=\dfrac{1}{2}-\dfrac{2}{5k}$.

Moreover
\begin{equation*}
\begin{split}
\|D^{s_k}_xu^{k-j}v^j\|_{L^{p_0}_xL^{q_0}_I}&\leq c \|D^{s_k}_xu^{k-j}\|_{L^{p_j}_xL^{q_j}_I}\|v\|^{j}_{L^{5k/4}_xL^{5k/2}_I}\\
&+\|u\|^{k-j}_{L^{5k/4}_xL^{5k/2}_I}\|D^{s_k}_xv^{j}\|_{L^{\widetilde{p}_j}_xL^{\widetilde{q}_j}_I},
\end{split}
\end{equation*}
where 
$$
\dfrac{1}{p_j}=1-(j+1)\dfrac{4}{5k}, \,\,\,\dfrac{1}{q_j}=\dfrac{1}{2}-(j+1)\dfrac{2}{5k}
$$
and
$$
\dfrac{1}{\widetilde{p}_j}=1-(1+k-j)\dfrac{4}{5k}, \,\,\, \dfrac{1}{\widetilde{q}_j}=\dfrac{1}{2}-(1+k-j)\dfrac{2}{5k}.
$$

On the other hand
$$
\|D^{s_k}_xu^{k-j}\|_{L^{p_j}_xL^{q_j}_I}\leq c \|D^{s_k}_xu\|_{L^{5}_xL^{10}_I}\|u\|^{k-j-1}_{L^{5k/4}_xL^{5k/2}_I}
$$
and
$$
\|D^{s_k}_xv^{j}\|_{L^{\widetilde{p}_j}_xL^{\widetilde{q}_j}_I}\leq c \|D^{s_k}_xv\|_{L^{5}_xL^{10}_I}\|v\|^{j-1}_{L^{5k/4}_xL^{5k/2}_I}.
$$

Finally, collecting the above estimates we conclude
\begin{equation*}
\begin{split}
\vertiii{\Phi(u)-\Phi(v)} &\leq c\|D^{s_k}_x(u-v)\|_{L^{5}_xL^{10}_I}(\|u\|^k_{L^{5k/4}_xL^{5k/2}_I}+\|v\|^k_{L^{5k/4}_xL^{5k/2}_I})\\
&+\|u-v\|^k_{L^{5k/4}_xL^{5k/2}_I}(\|u\|^{k-1}_{L^{5k/4}_xL^{5k/2}_I}\\
&+\|v\|^{k-1}_{L^{5k/4}_xL^{5k/2}_I})(\|D^{s_k}_xu\|_{L^{5}_xL^{10}_I}+\|D^{s_k}_xv\|_{L^{5}_xL^{10}_I}).
\end{split}
\end{equation*}

Therefore, choosing $a$ and $b$ such that $b=2cK$, $\delta=a/2$, $ca^k\leq1/4$ and $ca^{k-1}b\leq1/4$, we establish the contraction property. Finally, since $b=2cK$ and the solution $u$ belongs to $X^k_{a,b}$ we show the inequalities \eqref{i0}. Moreover, by \eqref{HSP}, Lemma \ref{lemma1} $(iii)$ and \eqref{LocEst}, $u\in C(I;\dot{H}^{s_k}(\R))$ with its norm bounded by $b=2cK$.
\end{proof}

\begin{remark} \label{RemMax}
We can define the maximal interval of existence for any solution $u$ of \eqref{gkdv} obtained from Theorem \ref{local}. Indeed, suppose that $u^{(1)}, u^{(2)} \in C(I; \dot{H}^{s_{k}}(\R))$ are two solutions of \eqref{gkdv} on the closed interval $I$ with $u^{(1)}(t_{0}) = u_{0} = u^{(2)}(t_{0})$ for some $t_{0} \in I$. We claim that $u^{(1)} \equiv u^{(2)}$ on $I \times \R$. To see this, let
\begin{equation*}
K = \sup_{t \in I} \max_{i=1,2} \| u^{(i)} \|_{\dot{H}^{s_{k}}}.
\end{equation*}
Since $\| V(t - t_{0}) u_{0} \|_{L^{5k/4}_{x} L^{5k/2}_{t}} \leq c \| u_{0} \|_{\dot{H}^{s_{k}}}$, there exists and interval $\widetilde{I} \subseteq I$ such that $t_{0} \in \widetilde{I}$ and $\| V(t - t_{0}) u_{0} \|_{L^{5k/4}_{x} L^{5k/2}_{\widetilde{I}}} \leq \delta(K)$, where $\delta(K)$ is given by Theorem \ref{local}. By choosing a smaller interval $\widetilde{I}$, if necessary, we can also assume that for $i = 1,2$ we have
\begin{equation}
\label{I12}
\| u^{(i)} \|_{L^{5k/4}_{x} L^{5k/2}_{\widetilde{I}}} \leq a \qquad \text{and} \qquad \| D^{s_{k}}_{x} u^{(i)} \|_{L^{5}_{x} L^{10}_{\widetilde{I}}} \leq b,
\end{equation}
where $a$ and $b$ are obtained in the proof of Theorem \ref{local}. The uniqueness of the fixed point in $X^{k}_{a,b}$ gives us that $u^{(1)} \equiv u^{(2)}$ on $\widetilde{I} \times \R$. Because we can partition the interval $I$ into a finite collection of subintervals $I_{j}$, each of which satisfy the inequalities \eqref{I12}, a continuation argument gives $u^{(1)} \equiv u^{(2)}$ on $I \times \R$.

In view of the above computations, we can define a maximal interval $I(u_{0}) = (t_{0} - T_{-}(u_{0}), t_{0} + T_{+}(u_{0}))$ with $T_{+}(u_{0}), T_{-}(u_{0}) > 0$, where the solution of \eqref{gkdv} with initial data $u(t_{0}) = u_{0}$ is defined. Furthermore, if $T_{1} < t_{0} + T_{+}(u_{0})$ and $T_{2} > t_{0} - T_{-}(u_{0})$ and $T_{2} < t_{0} < T_{1}$, then $u$ solves \eqref{gkdv} in $[T_{2}, T_{1}] \times \R$ with initial data $u(t_{0}) = u_{0}$ and $u \in C([T_{2}, T_{1}]; \dot{H}^{s_{k}}(\R)), u \in L^{5k/4}_{x} L^{5k/2}_{[T_{2}, T_{1}]},$ and $D^{s_{k}}_{x} u \in L^{5}_{x} L^{10}_{[T_{2}, T_{1}]}$.
\end{remark}

As a consequence of Theorem \ref{local} we have the following result.

\begin{corollary}[Small Data Global Theory]\label{smallglobal}
\label{global}
Let $k \geq 4$ and $s_{k} = (k-4)/2k$. There exists $\delta_{k} > 0$ such that for any $u_{0} \in \dot{H}^{s_{k}}(\R)$ with $\| u_{0} \|_{\dot{H}^{s_{k}}} < \delta_{k}$, the corresponding solution $u$ of \eqref{gkdv} with $u(0) = u_{0}$ is global in time. Moreover,
\begin{equation*}
\| u \|_{L^{\infty}_{t} \dot{H}^{s_{k}}} + \| u \|_{L^{5k/4}_{x} L^{5k/2}_{t}} + \| D^{s_{k}}_{x} u \|_{L^{5}_{x} L^{10}_{t}} \leq 2c \| u_{0} \|_{\dot{H^{s_{k}}}}.
\end{equation*}
\end{corollary}

\begin{proof}
The result follows from the same calculations used in the proof of Theorem \ref{local}. Let
\begin{equation*}
\vertiii{u} :=  \| u \|_{L^{\infty}_{t} \dot{H}^{s_{k}}_{x}} + \| u \|_{L^{5k/4}_{x} L^{5k/2}_{t}} + \| D^{s_{k}}_{x} u \|_{L^{5}_{x} L^{10}_{t}}
\end{equation*}
and define
\begin{equation*}
X^{k}_{a} = \{ u: \R \times \R \to \R \ \vert \  \vertiii{u} \leq a \}.
\end{equation*}
Applying the estimates from Lemma \ref{lemma1} to \eqref{Phi} yields
\begin{equation*}
\vertiii{\Phi(u)} \leq c \| u_{0} \|_{\dot{H^{s_{k}}}} + c \| u \|^{k}_{L^{5k/4}_{x} L^{5k/2}_{t}} \| D^{s_{k}}_{x}u \|_{L^{5}_{x} L^{10}_{t}} \leq c \| u_{0} \|_{\dot{H}^{s_{k}}} + ca^{k+1}.
\end{equation*}
Choosing $a = 2c \| u_{0} \|_{\dot{H}^{s_{k}}}$ and $\| u_{0} \|_{\dot{H}^{s_{k}}}$ such that $2^{k} c^{k+1} \| u_{0} \|_{\dot{H}^{s_{k}}}^{k} < 1/2$ we obtain
\begin{equation*}
\vertiii{\Phi(u)} \leq 2c \| u_{0} \|_{\dot{H}^{s_{k}}}.
\end{equation*}
To see that $\Phi$ is a contraction one uses similar estimates and the proof is completed with standard arguments.
\end{proof}

\begin{remark} 
Taking into account the smallness assumption \eqref{SA} needed to have global solutions in $H^1(\R)$ it is reasonable to conjecture that $\delta_{4} = \|Q\|_{L^2}$ for the critical gKdV equation, where $Q$ is the unique positive radial solution of the elliptic equation \eqref{GSE} and this is an interesting
open problem (see Linares and Ponce \cite{LP09} chapter 8).
\end{remark}

\begin{corollary}\label{AnotherEst}
Let $k \geq 4,\ s_{k} = (k-4)/2k,\ u_{0} \in \dot{H}^{s_{k}}(\R)$ with $\| u_{0} \|_{\dot{H}^{s_{k}}} \leq K$, and $t_{0} \in I$, a time interval. Assume $u$ is a solution of the integral equation \eqref{IntEq} in $I \times \R$ with $u \in C(I; \dot{H}^{s_{k}}(\R))$ satisfying
\begin{equation*}
\| u \|_{L^{5k/4}_{x} L^{5k/2}_{I}} + \| D^{s_{k}}_{x} u \|_{L^{5}_{x} L^{10}_{I}} < M.
\end{equation*}
Then $u$ also satisfies
\begin{equation*}
\| D^{1+s_k}_xu \|_{L^{\infty}_xL^{2}_I} + \| u \|_{L_x^{\frac{k(3k-2)}{3k-4}}L^{3k-2}_{I}} < cK+cM^{k+1},
\end{equation*}
for some positive constant $c$.
\end{corollary}
\begin{proof}
This is a direct consequence of Lemma \ref{lemma12} and the estimate
\begin{equation*}
\begin{split}
\|D^{s_k}_x(u^{k+1})&\|_{L^1_xL^2_I}\leq c \|u\|^k_{L^{5k/4}_xL^{5k/2}_I}\|D_x^{s_k}u\|_{L^{5}_xL^{10}_I},
\end{split}
\end{equation*}
obtained in the proof of Theorem \ref{local}.
\end{proof}

\begin{remark}\label{AnotherEst2}
Let $\{U^j\}_{j\in \mathbb{N}}$ the family of nonlinear profiles given in Theorem \ref{Nprofdec}. Setting
\begin{equation*}
I^j_n=(h_n^{j})^3I_n+t_n^{j}
\end{equation*}
the assumption \eqref{(i)} implies $I^j_n\rightarrow I^j$ such that $U^j$ is well defined in $\overline{I^j}$ (the closure of $I^j$) and satisfies
\begin{equation*}
\| U^j \|_{L^{5k/4}_{x} L^{5k/2}_{\overline{I^j}}} + \| D^{s_{k}}_{x} U^j \|_{L^{5}_{x} L^{10}_{\overline{I^j}}} < \infty.
\end{equation*}
In view of Corollary \ref{AnotherEst} we also deduce
\begin{equation*}
\| D^{1+s_k}_xU^j \|_{L^{\infty}_xL^{2}_{\overline{I^j}}} + \|U^j \|_{L_x^{\frac{k(3k-2)}{3k-4}}L^{3k-2}_{\overline{I^j}}} < \infty
\end{equation*}
and
\begin{equation*}
 \lim_{n\rightarrow \infty} \left(\| D^{1+s_k}_xU_n^j \|_{L^{\infty}_xL^{2}_{I_n}} + \|U_n^j \|_{L_x^{\frac{k(3k-2)}{3k-4}}L^{3k-2}_{I_n}}\right) <\infty,
 \end{equation*} 
if the assumption \eqref{(i)} is satisfied.

Also note that if $I^j=\emptyset$ then \eqref{(i)} is automatically satisfied for this index $j$.
\end{remark}

We end this section with the following finite time blow-up criteria proved in \cite{FP13}.

\begin{theorem}[Theorem 1.4 in \cite{FP13}]
\label{BUR}
Assume $k \geq 4$ and $s_{k}= (k-4)/2k$. Suppose that $u_{0} \in \dot{H}^{s_{k}}$ and $t_{0} \geq 0$. Let $u$ be the solution of \eqref{gkdv} with initial data $u(t_{0}) = u_{0}$ and maximal interval of existence $I(u_{0}) = (t_{0} - T_{-}(u_{0}), t_{0} + T_{+}(u_{0}))$. If $T_{+}(u_{0}) < \infty$, then
\begin{equation}
\label{eq5}
\| u \|_{L^{5k/4}_{x} L^{5k/2}_{[t_{0}, t_{0} + T_{+}(u_{0}))}} = \infty.
\end{equation}
Moreover, if $\sup_{t \in [t_{0}, t_{0} + T_{+}(u_{0}))} \| u(t) \|_{\dot{H}^{s_{k}}} < \infty$, then
\begin{equation}
\label{eq6}
\| D^{2/3k}_{x} u \|_{L^{3k/2}_{x} L^{3k/2}_{[t_{0}, t_{0} + T_{+}(u_{0}))}} = \infty.
\end{equation}
An analogous statement holds for $T_{-}(u_{0})$.
\end{theorem}

\begin{proof}
The argument from \cite{FP13} yields the result in the case when $t_{0} = 0$ and is easily adapted to obtain the result stated above.
\end{proof}

\section{Nonlinear Profile Decomposition} \label{nlpd}

This section is devoted to the proof of Theorem \ref{Nprofdec}. Before proceeding to the proof of Theorem \ref{Nprofdec} we have the following important proposition.

\begin{proposition}[Existence of nonlinear profiles]
\label{existnlpd}
Suppose that $\psi \in \dot{H}^{s_{k}}(\R)$ and that $\{ t_{n} \}_{n \in \N}$ is a sequence with $\lim_{n \to \infty} t_{n} = \overline{t} \in [-\infty, \infty]$. Then there exists a nonlinear profile associated to $(\psi, \{ t_{n} \}_{n \in \N})$.
\end{proposition}

\begin{proof}
We consider two cases. If $\overline{t} \in (-\infty, \infty)$, then we use the arguments contained in Remark \ref{RemMax} to find an interval $\widetilde{I}$ containing $\overline{t}$ such that $u \in C(\widetilde{I}; \dot{H}^{s_{k}}(\R))$ solves \eqref{gkdv} with initial data $u_{0} = V(\overline{t}) \psi$. Therefore $u(t_{n}) \to V(\overline{t}) \psi$ as $n \to \infty$. Since we also have $V(t_{n})\psi \to V(\overline{t})\psi$ as $n \to \infty$ we obtain the desired result in this case.

Next suppose that $\overline{t} = +\infty$ (a similar argument applies in the case when $\overline{t} = -\infty$). We solve the integral equation
\begin{equation}
\label{IENLP}
u(t) = V(t)\psi - \int_{t}^{+\infty} V(t - t') \partial_{x}(u^{k+1})(t') dt'
\end{equation}
in $(t_{n_{0}}, +\infty) \times \R$ for $n_{0} \in \N$ large and satisfying
\begin{equation*}
\| V(t) \psi \|_{L^{5k/4}_{x} L^{5k/2}_{[t_{n_{0}}, +\infty)}} < \delta,
\end{equation*}
where $\delta$ is given by Theorem \ref{local}. The same argument that was used in the proof of Theorem \ref{local} now allows us to construct a solution to the integral equation \eqref{IENLP} such that $u \in C([t_{n_{0}}, +\infty); \dot{H}^{s_{k}}(\R))$, $u \in L^{5k/4}_{x} L^{5k/2}_{[t_{n_{0}},+\infty)}$, $D^{s_{k}}_{x} u \in L^{5}_{x} L^{10}_{[t_{n_{0}}, +\infty)}$, and
\begin{equation}
\label{IENLP2}
\| D^{s_{k}}_{x} (u^{n+1}) \|_{L^{1}_{x} L^{2}_{[t_{n_{0}}, +\infty)}} < \infty.
\end{equation}
The Strichartz estimates from Lemma \ref{lemma1} imply that for $n$ sufficiently large we have
\begin{equation*}
\| u(\cdot, t_{n}) - V(t_{n}) \psi \|_{\dot{H}^{s_{k}}} \leq c \| D^{s_{k}}_{x} (u^{k+1}) \|_{L^{1}_{x} L^{2}_{[t_{n},+\infty)}},
\end{equation*}
which goes to zero as $n \to \infty$ by \eqref{IENLP2}.
\end{proof}

\begin{remark}\label{4.2}
Note that if $u^{(1)}, u^{(2)} \in C(I;\dot{H}^{s_{k}}(\R))$ are both nonlinear profiles associated to $(\psi, \{ t_{n} \}_{n \in \N})$ in an interval $I$ containing $\overline{t} = \lim_{n\to \infty} t_{n}$, then $u^{(1)} \equiv u^{(2)}$ on $I \times \R$. This statement is clear if $\overline{t} \in (-\infty, \infty)$ thanks to Remark \ref{RemMax}. If $\overline{t} = +\infty$ (the case when $\overline{t} = -\infty$ being similar), then $u^{(1)}$ and $u^{(2)}$ are solutions of the integral equation \eqref{IENLP} in $(a, +\infty) \times \R$ for some $a \in \R$. Furthermore,  we have $u^{(i)} \in L^{5k/4}_{x} L^{5k/2}_{(a,+\infty)}$ and $D^{s_{k}}_{x} u^{(i)} \in L^{5}_{x} L^{10}_{(a,+\infty)}$. Using the same arguments as those from Remark \ref{RemMax} we conclude the claim.

By this remark we can also define the maximal interval $I$ of existence for the nonlinear profile associated to $(\psi, \{ t_{n} \}_{n \in \N})$.
\end{remark}

\begin{proof}[Proof of Theorem \ref{Nprofdec}]
The proof is based on the ideas developed by Keraani \cite{Keraani2006} to obtain related results involving the $L^2$-critical NLS equation (see also \cite{K01}). 

\smallskip

\noindent{\textbf{Step 1.}} We begin by proving that \eqref{(i)} implies \eqref{(ii)}. Let
\begin{equation}\label{rnJ}
 r_n^J=u_n-\sum_{j=1}^{J}U_n^j - V(t)R_n^J.
\end{equation}

First note that, since $u_n$ and $U^j$ are solutions of the (gKdV) equation, $r_n^J$ satisfies the equation
\begin{equation}\label{rnJ2}
\begin{cases}
\partial_t r_n^J+\partial_x^3r_n^J +\partial_x(f_n^J)=0, \\
r_n^J(x,0)=\sum_{j=i}^{J}(V_n^j-U_n^j)(x,0),
\end{cases}
\end{equation}
where $U_n^j$ is given in relation \eqref{Unj},
\begin{equation*}
 V_n^j(x,t)=\dfrac{1}{(h_n^j)^{2/k}}V\left(\dfrac{t-t_n^j}{(h_n^j)^{3}}\right)\psi^j\left(\dfrac{x-x_n^j}{h_n^j}\right) 
\end{equation*}
and
\begin{equation*}
 f_n^J(x,t)=\left(\sum_{j=1}^{J}U_n^j(x,t) + V(t)R_n^J(x) + r_n^J(x,t)\right)^{k+1}-\sum_{j=1}^{J}\left(U_n^j(x,t)\right)^{k+1}.
\end{equation*}

Letting $I=[a,b]$, the integral equation associated to \eqref{rnJ2} with initial time $a$ is the following
\begin{equation*}
r_n^J(t)=V(t-a)r_n^J(a)-\int_{a}^tV(t-t')\partial_x(f_n^J)(t')dt',
\end{equation*}
for every $t\in [a,b]$.

Define 
$$
\vertiii{r_n^J}_I=\|r_n^J\|_{L^\infty_I\dot{H}_x^{s_k}}+\|r_n^J\|_{L^{5k/4}_xL^{5k/2}_I}+\|D^{s_k}_xr_n^J\|_{L^5_xL^{10}_I}.
$$
By the Strichartz estimates (Lemma \ref{lemma1}) we deduce
\begin{equation}\label{rnJ4}
\vertiii{r_n^J}_I\leq c(\|r_n^J(a)\|_{\dot{H}^{s_k}}+\|D^{s_k}_xf_n^J\|_{L^1_xL^{2}_I}).
\end{equation}
We claim that
\begin{equation}\label{rnJ3}
\lim_{n\rightarrow \infty}\vertiii{r_n^J}_{I_n}\rightarrow 0, \peq \textrm{ as } \peq J\rightarrow 0.
\end{equation}
If \eqref{rnJ3} is true, by \eqref{rnJ} and Strichartz estimates (Lemma \ref{lemma1}) we have
\begin{align*}
&\lim_{n\rightarrow \infty} (\|D_x^{s_k}u_n\|_{L^{5}_xL^{10}_{I_n}} +\|u_n\|_{L_x^{5k/4}L^{5k/2}_{I_n}}) \\
\leq &\sum_{j=1}^{J_0}\lim_{n\rightarrow \infty}(\|D_x^{s_k}U^j_n\|_{L^{5}_xL^{10}_{I_n}} +\|U^j_n\|_{L_x^{5k/4}L^{5k/2}_{I_n}})
+\sup_{n,J}\|R_n^J\|_{\dot{H}^{s_k}}+1,
\end{align*}
for some $J_0\in \N$. Therefore, the assumption \eqref{(i)} and the Pythagorean expansion \eqref{LPYTHA} imply \eqref{(ii)}.

Next, we prove the limit \eqref{rnJ3}. For every interval $I=[a,b]\subseteq I_n$, we estimate the second term in the right hand side of \eqref{rnJ4} by
\begin{align*}
&\|D^{s_k}_xf_n^J\|_{L^1_xL^{2}_I} \\ 
\leq & \left\|D^{s_k}_x\left(\left(\sum_{j=1}^{J}U_n^j\right)^{k+1}-\sum_{j=1}^{J}\left(U_n^j\right)^{k+1}\right)\right\|_{L^1_xL^{2}_{I}}\\
+ &\left\|D^{s_k}_x\left(\left(\sum_{j=1}^{J}U_n^j + V(t)R_n^J\right)^{k+1}-\left(\sum_{j=1}^{J}U_n^j\right)^{k+1}\right)\right\|_{L^1_xL^{2}_{I}}\\
+ &\left\|D^{s_k}_x\left(\left(\sum_{j=1}^{J}U_n^j + V(t)R_n^J + r_n^J(x,t)\right)^{k+1}-\left(\sum_{j=1}^{J}U_n^j+V(t)R_n^J\right)^{k+1}\right)\right\|_{L^1_xL^{2}_{I}}\\
\equiv \  & I_n^J+II_n^J+III_n^J
\end{align*}

The last term is the easiest one to handle. Indeed, using the same computations as in the proof of small data theory (Theorem \ref{local}) and combining with the last two inequalities, we obtain
\begin{equation}
\begin{aligned}
\label{rnJI}
\vertiii{r_n^J}_I
&\leq c(\|r_n^J(a)\|_{\dot{H}^{s_k}}+I_n^J+II_n^J)\\
&+c(\vertiii{r_n^J}^{k+1}_{I}+\vertiii{r_n^J}^{k}_{I}\|D^{s_k}_x(\sum_{j=1}^{J}U_n^j + V(t)R_n^J)\|_{L^{5}_xL^{10}_{I}})\\
&+c\vertiii{r_n^J}^{2}_{I}\|\sum_{j=1}^{J}U_n^j + V(t)R_n^J\|^k_{L^{5k/4}_xL^{5k/2}_{I}}\\
&+c\vertiii{r_n^J}_{I}\|\sum_{j=1}^{J}U_n^j + V(t)R_n^J\|^{k-1}_{L^{5k/4}_xL^{5k/2}_{I}}\|D^{s_k}_x(\sum_{j=1}^{J}U_n^j + V(t)R_n^J)\|_{L^{5}_xL^{10}_{I}}\\
&+c\vertiii{r_n^J}_{I}\|\sum_{j=1}^{J}U_n^j + V(t)R_n^J\|^{k}_{L^{5k/4}_xL^{5k/2}_{I}}.
\end{aligned}
\end{equation}

We claim that 
\begin{equation*}
\Lambda_{n}^{J} := \|\sum_{j=1}^{J}U_n^j + V(t)R_n^J\|_{L^{5k/4}_xL^{5k/2}_{I_n}}+\|D^{s_k}_x(\sum_{j=1}^{J}U_n^j + V(t)R_n^J)\|_{L^{5}_xL^{10}_{I_n}}
\end{equation*}
is uniformly bounded. Indeed, we have the following result.

\begin{lemma}\label{UnifB}
There exists $C>0$ such that
$\displaystyle \limsup_{J\rightarrow \infty} \left ( \limsup_{n\rightarrow \infty} \Lambda_{n}^{J} \right ) < C.$
\end{lemma}

\begin{proof}[Proof of Lemma \ref{UnifB}] 
We first note that by Strichartz estimates (Lemma \ref{lemma1})
\begin{equation}\label{RnJ}
\|V(t)R_n^J\|_{L^{5k/4}_xL^{5k/2}_{t}}+\|D^{s_k}_xV(t)R_n^J\|_{L^{5}_xL^{10}_{t}}
\leq c\|R_n^J\|_{\dot{H}^{s_k}}\leq c,
\end{equation}
for all $J,n\in \N$, since $\{\phi_n\}_{n\in \mathbb{N}}$ is a bounded sequence in $\dot{H}^{s_k}(\R)$ and the asymptotic Pythagorean expansion \eqref{LPYTHA} holds.

Therefore,  we just need to prove
\begin{equation}\label{LimUnj}
\limsup_{J\rightarrow \infty} \left ( \limsup_{n\rightarrow \infty}\|\sum_{j=1}^{J}U_n^j\|_{L^{5k/4}_xL^{5k/2}_{I_n}}+\|D^{s_k}_x(\sum_{j=1}^{J}U_n^j)\|_{L^{5}_xL^{10}_{I_n}} \right ) <+\infty.
\end{equation}

Using the pairwise orthogonality of $(h_n^j, x_n^j, t_n^j)_{n\in \mathbb{N}, j\in \mathbb{N}}$ given in Theorem \ref{Lprofdec} we can prove (see for instance Farah and Versieux \cite{FV15} Lemma 3.1)
\begin{equation}\label{sumUnj}
\limsup_{n\rightarrow \infty} \|\sum_{j=1}^{J}U_n^j\|^{5k/4}_{L^{5k/4}_xL^{5k/2}_{I_n}}
\leq \sum_{j=1}^{J} \|U_n^j\|^{5k/4}_{L^{5k/4}_xL^{5k/2}_{I_n}}, \peq \textrm{ for all} \peq  J\geq 1.
\end{equation}
A similar idea can be used to also deduce
\begin{equation}\label{sumUnj2}
\limsup_{n\rightarrow \infty} \|D^{s_k}_x(\sum_{j=1}^{J}U_n^j)\|^5_{L^{5}_xL^{10}_{I_n}}
\leq \sum_{j=1}^{J} \|D^{s_k}_xU_n^j\|^5_{L^{5}_xL^{10}_{I_n}}, \peq \textrm{ for all} \peq  J\geq 1.
\end{equation}

On the other hand, by Strichartz estimates (Lemma \ref{lemma1}) and the asymptotic Pythagorean expansion \eqref{LPYTHA}
\begin{equation}\label{psij}
\sum_{j\geq 1} \|V(t)\psi^j\|^2_{L^{5k/4}_xL^{5k/2}_{t}}+\sum_{j\geq 1}\|D^{s_k}_xV(t)\psi^j\|^2_{L^{5}_xL^{10}_{t}}
\leq c\sum_{j\geq 1}\|\psi^j\|^2_{\dot{H}^{s_k}}<+\infty.
\end{equation}

Let $\delta_0>0$ given by Definition \ref{DefDelta}. In view of \eqref{psij} there exists $J(\delta_0)>1$ sufficiently large such that
$$
\|\psi^j\|^2_{\dot{H}^{s_k}}<(\delta_0/2)^2 \peq \textrm{ for all} \peq  j\geq J(\delta_0).
$$

Also, by Definition \ref{NLP} and \eqref{HSP}, for every  $j\geq J(\delta_0)$ there exists $T^j\in \R$ such that
\begin{equation*}
\|U^j(T^j)\|^2_{\dot{H}^{s_k}}< \|\psi^j\|^2_{\dot{H}^{s_k}}+(\delta_0/2^{j+1})^2<\delta_0^2.
\end{equation*}

Using the small data global theory (Corollary \ref{global}) with $u_0=U^j(T^j)$ we deduce that $U^j$ is globally defined and
\begin{equation*}
\|U^j\|_{L^{5k/4}_xL^{5k/2}_t}+ \|D^{s_k}_xU^j\|_{L^5_xL^{10}_t}\leq 2c\|U^j(T^j)\|_{\dot{H}^{s_k}}.
\end{equation*}

Collecting the last two inequalities and \eqref{psij}, we have (since $5k/4>2$, for $k>4$)
\begin{equation}\label{sumUj}
\sum_{j\geq J(\delta_0)} \|U^j\|^{5k/4}_{L^{5k/4}_xL^{5k/2}_{t}}+\sum_{j\geq J(\delta_0)}\|D^{s_k}_xU^j\|^{5}_{L^{5}_xL^{10}_{t}} <+\infty.
\end{equation}

We still have to consider a finite number of nonlinear profiles $\{U^j\}_{1\leq j\leq J(\delta_0)}$, however by triangle inequality, for every $n\in \N$
\begin{equation}\label{sumUj2}
\begin{aligned}
& \|\sum_{j=1}^{J(\delta_0)}U_n^j\|_{L^{5k/4}_xL^{5k/2}_{I_n}}+\|D^{s_k}_x(\sum_{j=1}^{J(\delta_0)}U_n^j)\|_{L^{5}_xL^{10}_{I_n}}\\
< &\sum_{j=1}^{J(\delta_0)}\|U_n^j\|_{L^{5k/4}_xL^{5k/2}_{I_n}}+\|D^{s_k}_xU_n^j\|_{L^{5}_xL^{10}_{I_n}}.
\end{aligned}
\end{equation}

Moreover, assumption \eqref{(i)} of Theorem \ref{Nprofdec} imply that the right hand side of the above inequality is finite. 

Finally, in view of \eqref{sumUnj} and \eqref{sumUnj2}, the inequalities \eqref{sumUj} and \eqref{sumUj2} imply \eqref{LimUnj} and we complete the proof of Lemma \ref{UnifB}.
\end{proof}

Returning to the proof of Theorem \ref{Nprofdec}, by Lemma \ref{UnifB} we can bound the right hand side of inequality \eqref{rnJI} by
\begin{equation}\label{rnJI2}
\begin{split}
\vertiii{r_n^J}_I
&\leq c(\|r_n^J(a)\|_{\dot{H}^{s_k}}+I_n^J+II_n^J)\\
&+\vertiii{r_n^J}_{I} \|\sum_{j=1}^{J}U_n^j\|^{k-1}_{L^{5k/4}_xL^{5k/2}_{I}} +\sum_{l=2,k,k+1}\vertiii{r_n^J}^l_{I}. 
\end{split}
\end{equation}

The next two lemmas will help us to complete the proof.
\begin{lemma}\label{LimInJ}
$\displaystyle
\limsup_{n\rightarrow \infty} \peq I_n^J+II_n^J \rightarrow 0, \peq \textrm{ as } \peq  J\rightarrow +\infty
$
\end{lemma}

\begin{lemma}\label{PartIn}
For every $\varepsilon>0$, there exist $p\in\N$ (which depends on $\varepsilon$ but not on $n$ and $J$) and a partition of $I_n$
$$
\displaystyle I_n=\bigcup_{i=1}^p I^i_n
$$
such that
\begin{equation*}
\limsup_{n\rightarrow \infty}\|\sum_{j=1}^{J}U_n^j\|_{L^{5k/4}_xL^{5k/2}_{I^i_n}}+\|D^{s_k}_x(\sum_{j=1}^{J}U_n^j)\|_{L^{5}_xL^{10}_{I^i_n}}\leq \varepsilon,
\end{equation*}
for every $1\leq i \leq p$ and every $J\geq 1$.
\end{lemma}

We postpone the proofs of these lemmas for a moment, and continue with the argument for Theorem \ref{Nprofdec}. Without loss of generality, let us assume $I_n\subseteq \R_+$ (since $0\in I_n$ the general case can be completed by considering the positive $I_n\cap\R_+$ and negative $I_n\cap\R_-$ parts of $I_n$). Let $\varepsilon>0$ and consider the partition of $I_n$ given by Lemma \ref{PartIn}, that is
$$
\displaystyle I_n=\bigcup_{i=1}^p I^i_n=[0,a_n^1]\cup[0,a_n^2]\cup \cdots \cup [0,a_n^p).
$$
By inequality \eqref{rnJI2}, for all interval $I_n^i$, we have for $n$ and $J$ large
\begin{equation}\label{rnJI3}
\begin{split}
\vertiii{r_n^J}_{I_n^i}
\leq c(\|r_n^J(a_n^{i-1})\|_{\dot{H}^{s_k}}+I_n^J+II_n^J+\varepsilon\vertiii{r_n^J}_{I_n^i} +\sum_{l=2,k,k+1}\vertiii{r_n^J}^l_{I_n^i},
\end{split}
\end{equation}
where we have set $a_n^{0}=0$.

Choosing $\varepsilon>0$ such that $c\varepsilon<1/2$ we can absorb the linear term in the right hand side of \eqref{rnJI3} to obtain
\begin{equation*}
\begin{split}
\vertiii{r_n^J}_{I_n^i}
\leq c(\|r_n^J(a_n^{i-1})\|_{\dot{H}^{s_k}}+I_n^J+II_n^J+\sum_{l=2,k,k+1}\vertiii{r_n^J}^l_{I_n^i}.
\end{split}
\end{equation*}
In particular, for the interval $I^1_n=[0,a_n^1]$
\begin{equation*}
\begin{split}
\vertiii{r_n^J}_{I_n^1}
\leq c(\|r_n^J(0)\|_{\dot{H}^{s_k}}+I_n^J+II_n^J+\sum_{l=2,k,k+1}\vertiii{r_n^J}^l_{I_n^1}.
\end{split}
\end{equation*}

Since $r_n^J(0)=\sum_{j=i}^{J}(V_n^j-U_n^j)(0)$, by definition of the nonlinear profiles $U_n^i$, we have
$$
\lim_{n\rightarrow \infty} \|r_n^J(0)\|_{\dot{H}^{s_k}} \rightarrow 0, \peq \textrm{ as } \peq J\rightarrow \infty.
$$
Therefore, in view of Lemma \ref{LimInJ} and a bootstrap argument we deduce
$$
\lim_{n\rightarrow \infty} \vertiii{r_n^J}_{I_n^1} \rightarrow 0, \peq \textrm{ as } \peq J\rightarrow \infty.
$$
In particular,
$$
\lim_{n\rightarrow \infty} \|r_n^J(a_n^1)\|_{\dot{H}^{s_k}} \rightarrow 0, \peq \textrm{ as } \peq J\rightarrow \infty
$$

This last limit allows us to repeat the same argument on $I_n^2$. By iterating this process, we obtain
$$
\lim_{n\rightarrow \infty} \vertiii{r_n^J}_{I_n^i} \rightarrow 0, \peq \textrm{ as } \peq J\rightarrow \infty,
$$
for every $1\leq i \leq p$.

Since $p\in\N$ does not depend on $n$ and $J$ we deduce the limit \eqref{rnJ3}.

\smallskip 

\noindent{\textbf{Step 2.}} We now turn to the proof of \eqref{(ii)} implies \eqref{(i)}. Suppose that
\begin{equation*}
\lim_{n \to \infty} \left ( \| D^{s_{k}}_{x} u_{n} \|_{L^{5}_{x} L^{10}_{I_{n}}} + \| u_{n} \|_{L^{5k/4}{x} L^{5k/2}_{I_{n}}} \right ) < \infty
\end{equation*}
and yet \eqref{(i)} fails. This means that there is a smallest $j_{0} \geq 1$ such that
\begin{equation*}
\lim_{n \to \infty} \left ( \| D^{s_{k}}_{x} U_{n}^{j_{0}} \|_{L^{5}_{x} L^{10}_{I_{n}}} + \| U_{n}^{j_{0}} \|_{L^{5k/4}_{x} L^{5k/2}_{I_{n}}} \right ) = \infty. 
\end{equation*}
It follows that
\begin{equation*}
\lim_{n \to \infty} \| D^{s_{k}}_{x} U_{n}^{j_{0}} \|_{L^{5}_{x} L^{10}_{I_{n}}} = \infty \qquad \text{or} \qquad \lim_{n \to \infty} \| U_{n}^{j_{0}} \|_{L^{5k/4}_{x} L^{5k/2}_{I_{n}}} = \infty. 
\end{equation*}

Assume first that
\begin{equation*}
\lim_{n \to \infty} \| U_{n}^{j_{0}} \|_{L^{5k/4}_{x} L^{5k/2}_{I_{n}}} = \infty. 
\end{equation*}
Since $j_{0}$ is the smallest positive integer for which this is true, we obtain
\begin{equation*}
\lim_{n \to \infty} \left \| \sum_{j=1}^{j_{0}} U_{n}^{j} + V(t)R_{n}^{j_{0}} + r_{n}^{j_{0}} \right \|_{L^{5k/4}_{x} L^{5k/2}_{I_{n}}} = \infty.
\end{equation*}
In the last line we have used
\begin{equation*}
\sup_{n} \| V(t) R_{n}^{j_{0}} \|_{L^{5k/4}_{x} L^{5k/2}_{I_{n}}} < \infty \qquad \text{and} \qquad \sup_{n} \| r_{n}^{j_{0}} \|_{L^{5k/4}_{x} L^{5k/2}_{I_{n}}} < \infty,
\end{equation*}
which is true by \eqref{LSPWNL2} and \eqref{NSPWNL}.

Moreover, by \eqref{Decomp} we deduce
\begin{equation*}
u_{n} = \sum_{j=1}^{j_{0}} U_{n}^{j_{0}} + V(t) R_{n}^{j_{0}} + r_{n}^{j_{0}},
\end{equation*}
which implies
\begin{equation*}
\lim_{n \to \infty} \| u_{n} \|_{L^{5k/4}_{x} L^{5k/2}_{I_{n}}} = \infty,
\end{equation*}
a contradiction.

In case
\begin{equation*}
\lim_{n \to \infty} \| D^{s_{k}}_{x} U_{n}^{j_{0}} \|_{L^{5}_{x} L^{10}_{I_{n}}} = \infty,
\end{equation*}
we can repeat the preceding argument, recalling inequalities \eqref{NSPWNL} and \eqref{RnJ}. This completes the proof of Theorem \ref{Nprofdec}.
\end{proof}

Now, we prove Lemmas \ref{LimInJ} and \ref{PartIn}.

\begin{proof}[Proof of Lemma \ref{PartIn}] 
By the proof of Lemma \ref{UnifB}, for every $\varepsilon>0$ sufficiently small, there exists $J(\varepsilon)\geq 1$ such that
\begin{equation*}
\limsup_{n\rightarrow \infty} \left ( \|\sum_{j> J(\varepsilon)}U_n^j\|_{L^{5k/4}_xL^{5k/2}_{t}}+\
\|D^{s_k}_x(\sum_{j> J(\varepsilon)}U_n^j)\|_{L^{5}_xL^{10}_{t}} \right ) \leq \varepsilon/2.
\end{equation*}

It remains to consider a finite number of nonlinear profiles $\{U^j\}_{1\leq j\leq J(\varepsilon)}$. Let $I^j$ denotes the maximal time of existence of $U^j$. By a change of variables and assumption \eqref{(i)} of Theorem \ref{Nprofdec}, we have
$$
\|D_x^{s_k}U^j\|_{L^{5}_xL^{10}_{\frac{I_n-t_n^j}{(h_n^j)^3}}}+\|U^j\|_{L_x^{5k/4}L^{5k/2}_{\frac{I_n-t_n^j}{(h_n^j)^3}}} = \|D_x^{s_k}U_n^j\|_{L^{5}_xL^{10}_{I_n}}+\|U_n^j\|_{L_x^{5k/4}L^{5k/2}_{I_n}} <\infty.
$$

Therefore, there exists a closed interval $\widetilde{I}^j\subseteq I^j$ such that
\begin{equation}\label{Uj}
\|D_x^{s_k}U^j\|_{L^{5}_xL^{10}_{\widetilde{I}^j}}+\|U^j\|_{L_x^{5k/4}L^{5k/2}_{\widetilde{I}^j}}  <\infty.
\end{equation}
and 
$$
\dfrac{I_n-t_n^j}{(h_n^j)^3} \subseteq \widetilde{I}^j,
$$ 
for $n$ large.

By \eqref{Uj}, we can construct a partition of $\widetilde{I}^j= \bigcup_{i=1}^{p_j} \widetilde{I}^j_i$ satisfying
\begin{equation}\label{Uj2}
\|D_x^{s_k}U^j\|_{L^{5}_xL^{10}_{\widetilde{I}_i^j}}+\|U^j\|_{L_x^{5k/4}L^{5k/2}_{\widetilde{I}_i^j}}  < \varepsilon/2J(\varepsilon),
\end{equation}
for every $1\leq i\leq p_j$.

Writing \eqref{Uj2} in terms of $U^j_n$ we have
$$
\|D_x^{s_k}U_n^j\|_{L^{5}_xL^{10}_{\widetilde{I}_{n,i}^j}}+\|U_n^j\|_{L_x^{5k/4}L^{5k/2}_{\widetilde{I}_{n,i}^j}}  < \varepsilon/2J(\varepsilon),
$$
where $\widetilde{I}_{n,i}^j=(h_n^j)^3\widetilde{I}_{i}^j+t_n^j$.

Taking ${I}_{n,i}^j=I_n\cap \widetilde{I}_{n,i}^j$ we construct the partial partition for all $1\leq j\leq J(\varepsilon)$. Finally, intersecting all the partial partitions we obtain the desired final partition, which is independent of $n$ and $J$.
\end{proof}

\begin{proof}[Proof of Lemma \ref{LimInJ}] 
Recall that
\begin{equation*}
I_n^J\leq \left\|D^{s_k}_x\left(\left(\sum_{j=1}^{J}U_n^j\right)^{k+1}-\sum_{j=1}^{J}\left(U_n^j\right)^{k+1}\right)\right\|_{L^1_xL^{2}_{I_n}}
\end{equation*}
and
\begin{equation*}
II_n^J\leq \left\|D^{s_k}_x\left(\left(\sum_{j=1}^{J}U_n^j + V(t)R_n^J\right)^{k+1}-\left(\sum_{j=1}^{J}U_n^j\right)^{k+1}\right)\right\|_{L^1_xL^{2}_{I_n}}.
\end{equation*}

We first consider the term $I_n^J$, which is bounded by a sum of terms of the form (the quantity of terms depends only on $J$ and $k$, but not on $n$)
$$
T_n=\left\|D^{s_k}_x\left(U_n^{j_1}\dots U_n^{j_{k+1}}\right)\right\|_{L^1_xL^{2}_{I_n}},
$$
where not all $j_l$ are equal, say $j_k\neq j_{k+1}$. Using the fractional derivative rule (Lemma \ref{Leibniz}), the above expression can be bounded by
\begin{equation*}
\begin{split}
T_n&\leq \sum_{l=1}^{k-1}\left\|D^{s_k}_xU_n^{j_l}\right\|_{L^5_xL^{10}_{I_n}}\prod_{r\neq l}^{k-1}\left\|U_n^{j_r}\right\|_{L^{5k/4}_xL^{5k/2}_{I_n}}\left\|U_n^{j_k}U_n^{j_{k+1}}\right\|_{L^{5k/8}_xL^{5k/4}_{I_n}}\\
&\qquad +\prod_{r=1}^{k-1}\left\|U_n^{j_r}\right\|_{L^{5k/4}_xL^{5k/2}_{I_n}}\left\|D^{s_k}_x(U_n^{j_k}U_n^{j_{k+1}})\right\|_{L^{p_k}_xL^{q_k}_{I_n}},
\end{split}
\end{equation*}
where 
\begin{equation*}
\frac{1}{p_k}=1-(k-1)\frac{4}{5k}
\qquad \text{and} \qquad  \frac{1}{q_k}=\frac{1}{2}-(k-1)\frac{2}{5k}.
\end{equation*}

To obtain the desired result we need to prove (recall assumption $(i)$ in Theorem \ref{Nprofdec})
\begin{equation}\label{Prodjk}
\left\|U_n^{j_k}U_n^{j_{k+1}}\right\|_{L^{5k/8}_xL^{5k/4}_{I_n}}+\left\|D^{s_k}_x(U_n^{j_k}U_n^{j_{k+1}})\right\|_{L^{p_k}_xL^{q_k}_{I_n}} \rightarrow 0, \peq \textrm{ as } \peq  n\rightarrow \infty.
\end{equation}

Next, we prove that the first term in the left hand side of the last relation goes to zero as $n\rightarrow \infty$. Indeed, since $((h_n^j)_{n\in \N}, (x_n^j)_{n\in \N},  (t_n^j)_{n\in \N})_{j\in \N}$ are pairwise orthogonal (see relation \eqref{LXT}), we have either
\begin{equation}\label{hnxntn}
\lim_{n\rightarrow \infty}
\frac{h^{j_k}_n}{h^{j_{k+1}}_n} + \frac{h^{j_{k+1}}_n}{h^{j_k}_n} =+\infty
\end{equation}
or
\begin{equation}\label{hnxntn2}
h^{j_k}_n=h^{j_{k+1}}_n \,\, \textrm{ and } \,\, \lim_{n\rightarrow \infty} \left|\dfrac{x_n^{j_k}-x_n^{j_{k+1}}}{h_n^{j_k}}\right|+\left|\dfrac{t_n^{j_k}-t_n^{j_{k+1}}}{(h_n^{j_k})^{3}}\right|= \infty.
\end{equation}
By density we can suppose $U^{j_{k}}$ and $U^{j_{k+1}}$ are continuous and compactly supported.  If \eqref{hnxntn} holds,  without loss of generality, we assume
\begin{equation}\label{hnhl}
\frac{h^{j_k}_n}{h^{j_{k+1}}_n} \rightarrow \infty,
\end{equation}
as $n\rightarrow \infty$ (the other case is similar).
Using the change of variables $x=h_n^{j_k}y+x_n^{j_k}$ and $t=(h_n^{j_k})^3s+t_n^{j_k}$ we can rewrite the first term in the right hand side of \eqref{Prodjk} as 
\begin{align}\label{Orthj}
\nonumber
&\left\|U_n^{j_k}U_n^{j_{k+1}}\right\|_{L^{5k/8}_{x}L^{5k/4}_{I_n}}\\
\nonumber
= &\left(\int_{-\infty}^{\infty} \left(\int_{I_n}  
|U^{j_k}_n(x,t)U^{j_{k+1}}_n(x,t)|^{5k/4}dt\right)^{1/2}\!\!\!\!\! dx\right)^{8/5k} \\
= &\left(\frac{h_n^{j_k}}{h_n^{j_{k+1}}}\right)^{2/k}\!\!\left(\int_{-\infty}^{\infty} \left(\int_{I_n^k}  \left|U^{j_k}(y,s)U^{j_{k+1}}\left(y_k, s_k\right)\right|^{5k/4}ds\right)^{1/2}\!\!\!\!\!dy\right)^{8/5k}\!\!\!\!\!\!\!\!\!\!,
\end{align}
where 
\begin{equation*}
y_k=\frac{h_n^{j_{k}}}{h_n^{j_{k+1}}}y+\dfrac{x_n^{j_{k}}-x_n^{j_{k+1}}}{h_n^{j_{k}}}, \qquad  s_k=\left(\frac{h_n^{j_{k}}}{h_n^{j_{k+1}}}\right)^3s+\dfrac{t_n^{j_{k}}-t_n^{j_{k+1}}}{(h_n^{j_{k+1}})^3},
\end{equation*}
\text{and}
\begin{equation*}
I_n^k=\frac{I_n-t_n^{j_k}}{(h_n^{j_k})^3} \subset I^{j_k}\cap I^{j_{k+1}}
\end{equation*}
(here $I^l$ denotes the maximal interval of existence for $U^l$).
Since $U^{j_{k}}$ and $U^{j_{k+1}}$ are compactly supported, we obtain
$$
\left\|U_n^{j_k}U_n^{j_{k+1}}\right\|_{L^{5k/8}_{x}L^{5k/4}_{I_n}}\leq c \left(\frac{h_n^{j_k}}{h_n^{j_{k+1}}}\right)^{2/k},
$$
which implies the desired result by the assumption \eqref{hnhl}.

Now, assume that \eqref{hnxntn2} holds. Since $U^{j_{k}}$ and $U^{j_{k+1}}$ are continuous and compactly supported, in view of \eqref{Orthj} with $h^{j_k}_n=h^{j_{k+1}}_n$, we can apply the Lebesgue's Dominated Convergence to conclude that \eqref{Orthj} goes to zero as $n\rightarrow \infty$. Therefore, we have proved
\begin{equation}\label{LR}
\left\|U_n^{j_k}U_n^{j_{k+1}}\right\|_{L^{5k/8}_xL^{5k/4}_{I_n}} \rightarrow 0, \peq \textrm{ as } \peq  n\rightarrow \infty.
\end{equation}

Next, we treat the other norm in the left hand side of \eqref{Prodjk}. First note, by Remark \ref{AnotherEst2}, that
$$
\|D_x^{1+s_k}U_n^j\|_{L^{\infty}_xL^{2}_{I_n}}+\|U_n^j\|_{L_x^{\frac{k(3k-2)}{3k-4}}L^{3k-2}_{I_n}}<\infty,
$$
for every $j\geq 1$ and independent of $n\in \N$.
Therefore, setting $\sigma_k=\frac{k-4}{k(3k-2)}<1$ (for $k>4$) we obtain by interpolation
$$
\|D_x^{\sigma_k+s_k}U_n^j\|_{L^{a}_xL^{b}_{t}}\leq \|D_x^{1+s_k}U_n^j\|^{\sigma_k}_{L^{\infty}_xL^{2}_{I_n}}\|D_x^{s_k}U_n^j\|^{1-\sigma_k}_{L^{5}_xL^{10}_{I_n}},
$$
where 
\begin{equation*}
\dfrac{1}{a}=\dfrac{1-\sigma_k}{5} \quad \text{and} \quad \dfrac{1}{b}=\dfrac{\sigma_k}{2}+\dfrac{1-\sigma_k}{10}.
\end{equation*}
In view of the fractional Leibniz rule, Lemma \ref{Leibniz}, we deduce
\begin{equation}\label{LR1}
\begin{split}
\left\|D^{\sigma_k+s_k}_x(U_n^{j_k}U_n^{j_{k+1}})\right\|_{L^{\alpha}_xL^{\beta}_{I_n}}&\leq c 
\left\|D^{\sigma_k+s_k}_xU_n^{j_k}\right\|_{L^{a}_xL^{b}_{I_n}}
\left\|U_n^{j_{k+1}}\right\|_{L^{\frac{k(3k-2)}{3k-4}}_xL^{3k-2}_{I_n}} \\
&+c
\left\|D^{\sigma_k+s_k}_xU_n^{j_{k+1}}\right\|_{L^{a}_xL^{b}_{I_n}}
\left\|U_n^{j_{k}}\right\|_{L^{\frac{k(3k-2)}{3k-4}}_xL^{3k-2}_{I_n}}\\
&<\infty
\end{split}
\end{equation}
for every $j\geq 1$ and independent of $n\in \N$, where $$\dfrac{1}{\alpha}=\dfrac{1}{a}+\dfrac{3k-4}{k(3k-2)} \quad \text{and} \quad \dfrac{1}{\beta}=\dfrac{1}{b}+\dfrac{1}{3k-2}.$$

Finally, by interpolation we deduce
\begin{equation}\label{LR2}
\left\|D^{s_k}_x(U_n^{j_k}U_n^{j_{k+1}})\right\|_{L^{p_k}_xL^{q_k}_{I_n}}\leq \left\|U_n^{j_k}U_n^{j_{k+1}}\right\|^{\frac{\sigma_k}{\sigma_k+s_k}}_{L^{5k/8}_xL^{5k/4}_{I_n}} \left\|D^{\sigma_k+s_k}_x(U_n^{j_k}U_n^{j_{k+1}})\right\|^{\frac{s_k}{\sigma_k+s_k}}_{L^{\alpha}_xL^{\beta}_{I_n}},
\end{equation}
since 
$$
\dfrac{1}{p_k}=\left(\dfrac{\sigma_k}{\sigma_k+s_k}\right)\dfrac{8}{5k}+\left(\dfrac{s_k}{\sigma_k+s_k}\right)\dfrac{1}{\alpha}
$$ 
and
$$
\dfrac{1}{q_k}=\left(\dfrac{\sigma_k}{\sigma_k+s_k}\right)\dfrac{4}{5k}+\left(\dfrac{s_k}{\sigma_k+s_k}\right)\dfrac{1}{\beta}.
$$
Thus, in view of \eqref{LR2}, \eqref{LR1} and \eqref{LR} we conclude the proof of \eqref{Prodjk}.

Now, we consider the term $II_n^J$. To simplify the notation denote 
\begin{equation*}
W_n^J=\sum_{j=1}^JU_n^j. 
\end{equation*}
It is clear that
$$
(V(t)R_n^J+W_n^J)^{k+1}-(W_n^J)^{k+1}=\sum_{j=0}^kc_j(V(t)R_n^J)^{k+1-j}(W_n^J)^{j}.
$$
Therefore the term $II_n^J$ can be bounded by a sum of terms of the form (the quantity of terms depends only on $k$, but not on $n$ and $J$)
$$
S_n=\left\|D^{s_k}_x\left(f_n^1\dots f_n^{k}(V(t)R_n^J)\right)\right\|_{L^1_xL^{2}_{I_n}},
$$
where $f_n^l$ are equal to $V(t)R_n^J$ or $W_n^J$.

Again, using the fractional derivative rule (Lemma \ref{Leibniz}), we have the bound
\begin{equation*}
\begin{split}
S_n&\leq \sum_{l=1}^{k-1}\left\|D^{s_k}_xf_n^{l}\right\|_{L^5_xL^{10}_{I_n}}\prod_{r\neq l}^{k-1}\left\|f_n^{r}\right\|_{L^{5k/4}_xL^{5k/2}_{I_n}}\left\|f_n^{k}(V(t)R_n^J)\right\|_{L^{5k/8}_xL^{5k/4}_{I_n}}\\
&+\prod_{r=1}^{k-1}\left\|f_n^{r}\right\|_{L^{5k/4}_xL^{5k/2}_{I_n}}\left\|D^{s_k}_x(f_n^{k}(V(t)R_n^J))\right\|_{L^{p_k}_xL^{q_k}_{I_n}},
\end{split}
\end{equation*}
where 
\begin{equation*}
\frac{1}{p_k}=1-(k-1)\frac{4}{5k} \qquad \text{and} \qquad \frac{1}{q_k}=\frac{1}{2}-(k-1)\frac{2}{5k}
\end{equation*}

By relations \eqref{RnJ} and \eqref{LimUnj}, Holder's inequality and fractional derivative rule (Lemma \ref{Leibniz})
$$
S_n \leq c \left\|V(t)R_n^J\right\|_{L^{5k/4}_xL^{5k/2}_{I_n}} +c \left\|f_n^{k}D^{s_k}_x(V(t)R_n^J)\right\|_{L^{p_k}_xL^{q_k}_{I_n}}.
$$
In view of \eqref{LSPWNL2}, the desired result follows if we prove
\begin{equation}\label{fnkD}
\limsup_{n\rightarrow \infty} \left\|f_n^{k}D^{s_k}_x(V(t)R_n^J)\right\|_{L^{p_k}_xL^{q_k}_{I_n}} \rightarrow 0,  \peqq \textrm{as} \peqq J\rightarrow \infty.
\end{equation}

If $f_n^{k}=V(t)R_n^J$ it is true by Holder's inequality and \eqref{LSPWNL2}. So the interesting case is when $f_n^{k}=W_n^J$.
For all $J>J_0\geq 1$, by inequality \eqref{RnJ} we have
\begin{equation}\label{WnJD}
\left\|W_n^JD^{s_k}_x(V(t)R_n^J)\right\|_{L^{p_k}_xL^{q_k}_{I_n}} \leq \sum_{j=1}^{J_0}\left\|U_n^jD^{s_k}_x(V(t)R_n^J)\right\|_{L^{p_k}_xL^{q_k}_{I_n}} + c \left\|\sum_{j>J_0}^{J}U_n^j\right\|_{L^{5k/4}_xL^{5k/2}_{I_n}}
\end{equation}

In \eqref{sumUj}, we proved that $U^j$ is globally defined for every $j$ sufficiently large and moreover, for every $\varepsilon>0$ there exists $J(\varepsilon)\geq 1$ such that
$$
\sum_{j\geq J(\varepsilon)}\left\|U^j\right\|^{5k/4}_{L^{5k/4}_xL^{5k/2}_{t}}\leq \varepsilon^{5k/4}.
$$

By \eqref{sumUnj} and the fact that $\left\|U_n^j\right\|^{5k/4}_{L^{5k/4}_xL^{5k/2}_{t}}=\left\|U^j\right\|^{5k/4}_{L^{5k/4}_xL^{5k/2}_{t}}$, for every $J\geq J(\varepsilon)$ we deduce
$$
\limsup_{n\rightarrow \infty} \left\|\sum_{j>J(\varepsilon)}^{J}U_n^j\right\|_{L^{5k/4}_xL^{5k/2}_{I_n}}\leq \sum_{j\geq J(\varepsilon)}\left\|U^j\right\|^{5k/4}_{L^{5k/4}_xL^{5k/2}_{t}}\leq \varepsilon^{5k/4}.
$$
Thus, by \eqref{WnJD}, we have to prove that
\begin{equation}\label{UnjD}
\limsup_{n\rightarrow \infty} \left\|U_n^{j}D^{s_k}_x(V(t)R_n^J)\right\|_{L^{p_k}_xL^{q_k}_{I_n}} = 0,
\end{equation}
for every $1\leq j\leq J(\varepsilon)$ and since $\varepsilon>0$ is arbitrary we conclude \eqref{fnkD}.
Applying the change of variables $x=h_n^{j}y+x_n^{j}$ and $t=(h_n^{j})^3s+t_n^{j}$ we have
\begin{equation}\label{UnjD2}
\left\|U_n^{j} D^{s_k}_x(V(t)R_n^J)\right\|_{L^{p_k}_xL^{q_k}_{I_n}}=\left\|U^{j}D^{s_k}_xw_n^J\right\|_{L^{p_k}_xL^{q_k}_{I^j_n}},
\end{equation}
where $\displaystyle I_n^j=\frac{I_n-t_n^{j}}{(h_n^{j})^3} $ and $w_n^J(s,y) = (h_n^j)^{2/k}V((h_n^{j})^3s+t_n^{j})R_n^J(h_n^{j}y+x_n^{j})$.

A simple computation reveals
\begin{equation}\label{wnj}
\left\|w_n^J\right\|_{L^{5k/4}_xL^{5k/2}_{t}}=\left\|V(t)R_n^J\right\|_{L^{5k/4}_xL^{5k/2}_{t}}
\end{equation}
and 
\begin{equation}\label{Dwnj}
\left\|D^{s_k}_xw_n^J\right\|_{L^{5}_xL^{10}_{t}}=\left\|D^{s_k}_x(V(t)R_n^J)\right\|_{L^{5}_xL^{10}_{t}}.
\end{equation}
We claim that
\begin{equation}\label{LimDwnj}
\limsup_{n\rightarrow \infty} \left\|D^{s_k}_xw_n^J\right\|_{L^{\widetilde{p}}_xL^{\widetilde{q}}_{t}} \rightarrow 0,  \peqq \textrm{as} \peqq J\rightarrow \infty,
\end{equation}
for some $\widetilde{p}$ close to $5$ and $\widetilde{q}$ close to $10$.

Assuming the limit \eqref{LimDwnj} for a moment, let us conclude the proof of \eqref{UnjD}. Given $\varepsilon>0$, let $U^j_{\varepsilon}\in C_0^{\infty}(\R^2)$ such that
$$
\left\|U^j-U^j_{\varepsilon}\right\|_{L^{5k/4}_xL^{5k/2}_{\bar{I}^j}}<\varepsilon,
$$
where $\displaystyle \bar{I}^j=\lim_{n\rightarrow \infty} \frac{I_n-t_n^{j}}{(h_n^{j})^3}$.

Therefore, using relations \eqref{UnjD2}, \eqref{wnj} and \eqref{Dwnj} we have
\begin{align*}
&\left\|U_n^{j} D^{s_k}_x(V(t)R_n^J)\right\|_{L^{p_k}_xL^{q_k}_{I_n}}=\left\|U^{j}D^{s_k}_xw_n^J\right\|_{L^{p_k}_xL^{q_k}_{I^j_n}}\\
\leq &\left\|U^j-U^j_{\varepsilon}\right\|_{L^{5k/4}_xL^{5k/2}_{\bar{I}^j}}\left\|D^{s_k}_xw_n^J\right\|_{L^{5}_xL^{10}_{t}}+\left\|U_{\varepsilon}^{j}D^{s_k}_xw_n^J\right\|_{L^{p_k}_xL^{q_k}_{I^j_n}}\\
< &\varepsilon\left\|D^{s_k}_x(V(t)R_n^J)\right\|_{L^{5}_xL^{10}_{t}}+C_\varepsilon\left\|D^{s_k}_xw_n^J\right\|_{L^{\widetilde{p}}_xL^{\widetilde{q}}_{t}},
\end{align*}
where $C_\varepsilon=\left\|U^j_{\varepsilon}\right\|_{L^{\bar{p}}_xL^{\bar{q}}_{\bar{I}^j}}$, with 
\begin{equation*} \frac{1}{p_k}=\frac{1}{\widetilde{p}}+\frac{1}{\bar{p}} 
\qquad \text{and} \qquad
\frac{1}{q_k}=\frac{1}{\widetilde{q}}+\frac{1}{\bar{q}}.
\end{equation*}
Since $\varepsilon>0$ is arbitrary, by \eqref{RnJ} and \eqref{LimDwnj} we obtain \eqref{UnjD}.

To complete the proof we need to deduce \eqref{LimDwnj}. Let us first consider the case $k\geq 6$. Recall the sharp version of Kato's smoothing effect given in Lemma \ref{lemma12} $(i)$
\begin{equation}\label{KatoS}
\left\|D^{1+s_k}_xV(t)u_0\right\|_{L^{\infty}_xL^{2}_{t}}\leq c\|u_0\|_{\dot{H}^{s_k}}.
\end{equation}
Since $2/3k<s_k<1+s_k$ (for $k\geq 6$) interpolating inequalities \eqref{KatoS} and \eqref{STR} we can find $\theta \in (0,1)$, $a,b \in (1, \infty)$ such that
\begin{equation}\label{Interp}
\left\|D^{s_k}_xV(t)u_0\right\|_{L^{a}_xL^{b}_{t}}\leq c\|D^{2/3k}_xV(t)u_0\|_{L_{x,t}^{3k/2}}^{\theta}\left\|D^{1+s_k}_xV(t)u_0\right\|_{L^{\infty}_xL^{2}_{t}}^{1-\theta},
\end{equation}
where
\begin{equation*}
 \frac{2}{3k}\theta+(1+s_k)(1-\theta)=s_k, \quad \frac{1}{a}=\frac{2\theta}{3k}, \quad \text{and} \quad  \frac{1}{b}=\frac{2\theta}{3k}+\frac{1-\theta}{2}.
\end{equation*}

On the other hand, by Lemma \ref{lemma1}, we also have the Strichartz estimate 
$$
\|D^{s_k}_xV(t)u_0\|_{L^{5}_xL^{10}_t}\leq c\|u_0\|_{\dot{H}^{s_k}}.
$$
Interpolating again, we obtain for all $\delta\in (0,1)$
\begin{equation}\label{STR3}
\begin{split}
\left\|D^{s_k}_xV(t)u_0\right\|_{L^{\widetilde{p}_\delta}_xL^{\widetilde{q}_\delta}_{t}}&\leq c\|D^{s_k}_xV(t)u_0\|_{L^{5}_xL^{10}_t}^{\delta}\left\|D^{s_k}_xV(t)u_0\right\|_{L^{a}_xL^{b}_{t}}^{1-\delta}\\
&\leq c\|u_0\|_{\dot{H}^{s_k}}^{\delta + (1-\delta)(1-\theta)}\|D^{2/3k}_xV(t)u_0\|_{L_{x,t}^{3k/2}}^{\theta(1-\delta)},
\end{split}
\end{equation}
 where 
\begin{equation*} 
\frac{1}{\widetilde{p}_\delta}=\frac{\delta}{5}+\frac{1-\delta}{a}=\frac{\delta}{5}+\frac{2\theta(1-\delta)}{3k} \quad \text{and} \quad  \frac{1}{\widetilde{q}_\delta}=\frac{\delta}{10}+\frac{2\theta(1-\delta)}{3k}+\frac{(1-\theta)(1-\delta)}{2}.
\end{equation*}
A direct calculation yields
\begin{equation}\label{wnj2}
\|D^{2/3k}_xw_n^J\|_{L_{x,t}^{3k/2}}=\left\|D^{2/3k}_xV(t)R_n^J\right\|_{L_{x,t}^{3k/2}}
\end{equation}
and
\begin{equation}\label{wnj3}
\left\|D^{1+s_k}_xw_n^J\right\|_{L^{\infty}_xL^{2}_{t}}=\left\|D^{1+s_k}_x(V(t)R_n^J)\right\|_{L^{\infty}_xL^{2}_{t}}.
\end{equation}
Hence, combining \eqref{Interp}, \eqref{STR3}, \eqref{Dwnj}, \eqref{wnj2} and \eqref{wnj3} we obtain
$$
\left\|D^{s_k}_xw_n^J\right\|_{L^{\widetilde{p}_\delta}_xL^{\widetilde{q}_\delta}_{t}}\leq c\|R_n^J\|_{\dot{H}^{s_k}}^{\delta + (1-\delta)(1-\theta)}\|D^{2/3k}_xV(t)R_n^J\|_{L_{x,t}^{3k/2}}^{\theta(1-\delta)},
$$
for all $\delta\in (0,1)$.
By \eqref{RnJ} and \eqref{LSPWNL} (with $\displaystyle p=q=\frac{3k}{2}$), we deduce
$$
\limsup_{n\rightarrow \infty} \left\|D^{s_k}_xw_n^J\right\|_{L^{\widetilde{p}_\delta}_xL^{\widetilde{q}_\delta}_{t}} \rightarrow 0,  \peqq \textrm{as} \peqq J\rightarrow \infty.
$$
Finally, taking $\delta$ close to 1 we obtain that $\widetilde{p}_\delta$ is close to $5$ and $\widetilde{q}_\delta$ is close to $10$ which implies the claim \eqref{LimDwnj} in the case $k\geq 6$.

When $k=5$ we cannot apply the previous argument since $s_k< 2/3k$ in this case. However, if we replace the inequality \eqref{KatoS} by (see Lemma \ref{lemma12} $(i)$)
$$
\left\|D_x^{-1/k}V(t)u_0\right\|_{L^{k}_xL^{\infty}_{t}}\leq c\|u_0\|_{\dot{H}^{s_k}}
$$
we can carry out the same computations as before, observing that $-1/k <s_k< 2/3k$ for $k=5$, and also obtain the limit \eqref{LimDwnj}. This completes the proof of Lemma \ref{LimInJ}.

\end{proof}

\section{Concentration} \label{conc}

In this section we prove Theorem \ref{Concentration}.

\begin{proof}[Proof of Theorem \ref{Concentration}]
Let $u$ be a blowing up solution for \eqref{gkdv} at finite time $T^{\ast}<\infty$, and let $\{t_n\}_{n\in \N}$ be a sequence of times such that $t_{n} \rightarrow T^{\ast}$. Set
$$
u_n(x,t)=u(x, t_n+t).
$$
Since $u$ is defined in $[0,T^{\ast})$, $u_n$ is defined in $[-t_n,T^{\ast}-t_n)$. Also, the finite time blow-up criteria in Theorem \ref{BUR} yields
\begin{equation}\label{unblow}
\lim_{n\rightarrow \infty}\|u_n\|_{L^{5k/4}_xL^{5k/2}_{[0,T^{*}-t_n)}}=\lim_{n\rightarrow \infty}\|u_n\|_{L^{5k/4}_xL^{5k/2}_{[-t_n,0)}}=\infty.
\end{equation}

By the assumption \eqref{TypeII1}, the sequence $\{u_n(\cdot,0)\}_{n\in \N}=\{u(\cdot,t_n)\}_{n\in \N}$ is bounded in $\dot{H}^{s_k}(\R)$. Applying the Linear Profile Decomposition, Theorem \ref{Lprofdec}, for this sequence, we obtain (up to a subsequence) a sequence of functions $\{\psi^j\}_{j\in \mathbb{N}}\subset \dot{H}^{s_k}(\R)$ and sequences of parameters $(h_n^j, x_n^j, t_n^j)_{n\in \mathbb{N}, j\in \mathbb{N}}$ such that for every $J\geq1$ there exists $\{R^J_n\}_{n,J\in\mathbb{N}}\subset \dot{H}^{s_k}(\R)$ satisfying \eqref{SUM}, \eqref{LSPWNL2} and \eqref{LPYTHA}.

Considering the sequence of intervals $I_n=[0,T^{*}-t_n)$. In view of \eqref{unblow}, the Nonlinear Profile Decomposition, Theorem \ref{Nprofdec}, implies that there exists some $j_0\in \N$ such that $U^{j_0}$ (the nonlinear profile associated with $(\psi^{j_0}, \{-t^{j_0}_n/(h_n^{j_0})^3\}_{n\in \N})$) satisfies
\begin{equation}\label{Unblow}
\lim_{n\rightarrow \infty} \|D_x^{s_k}U_n^{j_0}\|_{L^{5}_xL^{10}_{I_n}}+\|U_n^{j_0}\|_{L_x^{5k/4}L^{5k/2}_{I_n}} =\infty.
\end{equation}

We claim that if $u$ is a solution for the gKdV equation \eqref{gkdv} such that, for some interval $S$, $\|u\|_{L_x^{5k/4}L^{5k/2}_{S}}<\infty$ then we must have $\|D_x^{s_k}u\|_{L^{5}_xL^{10}_{S}}<\infty$. Indeed, let $\varepsilon_0>0$ be an arbitrary small number. Since the norm $\|u\|_{L^{5k/4}_xL^{5k/2}_{S}}$ is finite we can find a partition of the interval $S$, namely $t_0<t_1<\ldots<t_\ell$, such that
$\|u\|_{L^{5k/4}_xL^{5k/2}_{S_n}}<\varepsilon_0$, where $S_n=[t_n,t_{n+1}]$,
$n=0,\dots,\ell-1$. Since $u$ is a solution of the integral equation
\eqref{IntEq}, from Lemma \ref{lemma1} $(ii)$ we deduce
\begin{equation*}
\begin{split}
\|D_x^{s_k}u\|_{L^5_xL^{10}_{S_n}}&\leq
c\| u_{0} \|_{\dot{H}^{s_{k}}}+c \left \|D_x^{s_k}\int_0^tU(t-t')\partial_x(u^{k+1})(t')dt' \right \|_{L^5_xL^{10}_{S_n}}\\
&\leq c\| u_{0} \|_{\dot{H}^{s_{k}}}+c\sum_{j=0}^n \left \|D_x^{s_k}\int_{t_j}^{t_{j+1}}U(t-t')\partial_x(u^{k+1})(t')dt' \right \|_{L^5_xL^{10}_{S_j}}\\
&\leq c\| u_{0} \|_{\dot{H}^{s_{k}}}+c\sum_{j=0}^n \left \|D_x^{s_k}\int_{0}^{t}U(t-t')\partial_x(u^{k+1})(t')\chi_{I_j}(t')dt' \right \|_{L^5_xL^{10}_{t}},\\
\end{split}
\end{equation*}
where $\chi_{S_j}$ denotes the characteristic function of the interval $S_j$. From Lemma \ref{lemma1} $(iii)$ and similar computations as the ones in the proof of Theorem \ref{local} we obtain
\begin{equation*}
\begin{split}
\|D_x^{s_k}u\|_{L^5_xL^{10}_{S_n}}&\leq
c\| u_{0} \|_{\dot{H}^{s_{k}}}+c\sum_{j=0}^n\|D_x^{s_k}u\|_{L^5_xL^{10}_{S_j}}\|u\|_{L^{5k/4}_xL^{5k/2}_{S_j}}^k\\
&\leq
c\| u_{0} \|_{\dot{H}^{s_{k}}}+c\varepsilon_0^k\sum_{j=0}^n\|D_x^{s_k}u\|_{L^5_xL^{10}_{S_j}}.
\end{split}
\end{equation*}
Therefore, choosing $c\varepsilon_0^k<1/2$, we conclude
\begin{equation}\label{eq9.3a}
\begin{split}
\|D_x^{s_k}u\|_{L^5_xL^{10}_{S_n}}\leq
2c\| u_{0} \|_{\dot{H}^{s_{k}}}+2\sum_{j=0}^{n-1}\|D_x^{s_k}u\|_{L^5_xL^{10}_{S_j}}.
\end{split}
\end{equation}
Inequality \eqref{eq9.3a} and an induction argument implies that
$\|D_x^{s_k}u\|_{L^5_xL^{10}_{S_n}}<\infty$ for $n=0,\dots,\ell-1$. By summing
over the $\ell$ intervals we conclude the claim.

Now, since
$$
\|D_x^{s_k}U_n^{j_0}\|_{L^{5}_xL^{10}_{I_n}}=\|D_x^{s_k}U^{j_0}\|_{L^{5}_xL^{10}_{I^{j_0}_n}}
$$
and 
$$
\|U_n^{j_0}\|_{L_x^{5k/4}L^{5k/2}_{I_n}}=\|U^{j_0}\|_{L_x^{5k/4}L^{5k/2}_{I^{j_0}_n}}
$$
where $I^{j_0}_n=\left[-\frac{t^{j_0}_n}{(h_n^{j_0})^3}, \frac{T^{\ast}-t_n}{(h_n^{j_0})^3}-\frac{t^{j_0}_n}{(h_n^{j_0})^3} \right)$, the condition \eqref{Unblow} implies
\begin{equation}\label{Unblow2}
\lim_{n\rightarrow \infty}\|U^{j_0}\|_{L_x^{5k/4}L^{5k/2}_{I^{j_0}_n}}=\infty
\end{equation}
In particular $\|\psi^{j_0}\|_{\dot{H}^{s_k}}\geq \delta_0$, otherwise $U^{j_0}$ is globally defined and then the above limit is false.

Let $t^{j_0}=\lim_{n\rightarrow \infty} -\frac{t^{j_0}_n}{(h_n^{j_0})^3}$. We have that $t^{j_0} \neq +\infty$, otherwise $I_n^{j_0}\rightarrow \emptyset$ (recall that $\{h_n^{j_0}\}_{n\in \N}$ is a sequence of positive numbers) and the limit \eqref{Unblow2} cannot be true in this case.

By the asymptotic Pythagorean expansion \eqref{LPYTHA} it follows that
\begin{equation}\label{PYTHALim}
\liminf_{n\rightarrow \infty} \|u(t_n)\|^2_{\dot{H}^{s_k}}\geq \sum_{j=1}^{\infty}\|\psi^j\|^2_{\dot{H}^{s_k}}.
\end{equation}

If there exists another $j_1\neq j_0$ such that \eqref{Unblow} holds then $\|\psi^{j_1}\|_{\dot{H}^{s_k}}\geq \delta_0$ and the inequality \eqref{PYTHALim} implies $\sup_{n\in \N} \|u(t_n)\|_{\dot{H}^{s_k}}\geq \sqrt{2}\delta_0$, which is a contradiction since we have assumed that  $u(t)\in C= \{ f \in \dot{H}^{s_{k}}(\R) \  \vert \ \delta_{0} \leq \| f \|_{\dot{H}^{s_{k}}} \leq (3\sqrt{2}/4) \delta_{0} \}$ for all $t\in [0,T^{\ast})$. Therefore, the profile $U^{j_0}$ obtained above is the only blowing up nonlinear profile.

Now, considering the sequence of intervals $I_n=[-t_n,0]$ and applying the same ideas we obtain
$$
\lim_{n\rightarrow \infty}\|U^{j_0}\|_{L_x^{5k/4}L^{5k/2}_{\bar{I}^{j_0}_n}}=\infty,
$$
where $\bar{I}^{j_0}_n=\left[-\frac{t_n}{(h_n^{j_0})^3}-\frac{t^{j_0}_n}{(h_n^{j_0})^3}, -\frac{t^{j_0}_n}{(h_n^{j_0})^3} \right)$.

Therefore, $t^{j_0} \neq -\infty$ and without loss of generality we can assume $t^{j_0} =0$. Thus, $U^{j_0}$ is a solution of the gKdV equation \eqref{gkdv} with initial data $\psi^{j_0}$ which blows up in positive and negative times. Moreover, denoting by $T^{\ast}_{j_0}>0$ the positive finite blow-up time of $U^{j_0}$, the limit \eqref{Unblow2} implies
\begin{equation}\label{Blowuptime}
\lim_{n\rightarrow \infty}\frac{T^{\ast}-t_n}{(h_n^{j_0})^3}\geq T^{\ast}_{j_0}.
\end{equation}

Returning to the linear profile decomposition for the sequence $\{u(\cdot,t_n)\}_{n\in \N}$, relation \eqref{SUM} and the change of variables $x=h_n^{j_0}y+x_n^{j_0}$ and $t=(h_n^{j_0})^3s+t_n^{j_0}$, imply that for every $J>j_0$
\begin{equation*}
(h_n^{j_0})^{2/k}V((h_n^{j_0})^3s+t_n^{j_0})u(h_n^{j_0}y+x_n^{j_0},t_n)=V\left(s\right)\psi^{j_0}\left(y\right)+\sum_{j\neq j_0}^{J}\widetilde{V}_n^j(y,s) + \widetilde{R}_n^J(y,s),
\end{equation*}
where
$$
\widetilde{V}_n^j(y,s)=\left(\dfrac{h_n^{j_0}}{h_n^j}\right)^{2/k}V\left(\left(\dfrac{h_n^{j_0}}{h_n^j}\right)^{3}s+\dfrac{t_n^{j_0}-t_n^j}{(h_n^j)^{3}}\right)\psi^j\left(\dfrac{h_n^{j_0}}{h_n^j}y+\dfrac{x_n^{j_0}-x_n^j}{h_n^j}\right)
$$
and 
$$
\widetilde{R}_n^J(y,s)=(h_n^{j_0})^{2/k}V((h_n^{j_0})^3s+t_n^{j_0})R_n^J(h_n^{j_0}y+x_n^{j_0}).
$$
Using the orthogonality of $(h_n^{j_0}, x_n^{j_0}, t_n^{j_0})_{n\in \N}$ and $(h_n^j, x_n^j, t_n^j)_{n\in \N}$ for $j\neq j_0$ it is easy to see that 
$$
\widetilde{V}_n^j \rightharpoonup 0 \textrm{ weakly in } L_x^{5k/4}L^{5k/2}_{t}.
$$
Moreover, since 
$$
\|\widetilde{R}_n^J\|_{L_x^{5k/4}L^{5k/2}_{t}}=\|V(t){R}_n^J\|_{L_x^{5k/4}L^{5k/2}_{t}}\leq c \|R_n^J\|_{\dot{H}^{s_k}}\leq c
$$
there exists $R^J \in \dot{H}^{s_k}(\R)$ such that
$$
\widetilde{R}_n^J \rightharpoonup R^J \textrm{ weakly in } L_x^{5k/4}L^{5k/2}_{t}.
$$
Therefore, for every $J>j_0$
$$
(h_n^{j_0})^{2/k}V((h_n^{j_0})^3s+t_n^{j_0})u(h_n^{j_0}y+x_n^{j_0},t_n)\rightharpoonup V\left(s\right)\psi^{j_0}\left(y\right)+R^J(y,s)
$$
weakly in $L_x^{5k/4}L^{5k/2}_{t}$.

On the other hand, 
$$
\|{R}^J\|_{L_x^{5k/4}L^{5k/2}_{t}}\leq \limsup_{n\rightarrow \infty} \|\widetilde{R}_n^J\|_{L_x^{5k/4}L^{5k/2}_{t}}=\|V(t){R}_n^J\|_{L_x^{5k/4}L^{5k/2}_{t}}\rightarrow 0,
$$
as $J\rightarrow \infty$. Therefore, by the uniqueness of the weak limit we conclude $R^J=0$ for every $J>j_0$ and
$$
(h_n^{j_0})^{2/k}V((h_n^{j_0})^3s+t_n^{j_0})u(h_n^{j_0}y+x_n^{j_0},t_n)\rightharpoonup V\left(s\right)\psi^{j_0}\left(y\right)\textrm{ weakly in } L_x^{5k/4}L^{5k/2}_{t}.
$$
A simple computation reveals
$$
(h_n^{j_0})^{2/k}V((h_n^{j_0})^3s+t_n^{j_0})u(h_n^{j_0}y+x_n^{j_0},t_n)=V(s)V\left(\frac{t^{j_0}_n}{(h_n^{j_0})^3}\right)f_n(y),
$$
where $f_n(y)=(h_n^{j_0})^{2/k}u(h_n^{j_0}y+x_n^{j_0},t_n)$.

At this point, we will make use of the following result.

\begin{lemma}\label{Vweak}
Let $\{\phi_n\}_{n\in \N}$ and $\phi$ be in $\dot{H}^{s_k}(\R)$. The following statements are equivalent:
\begin{itemize}
\item[$(i)$] $ \phi_n \rightharpoonup \phi \textrm{ weakly in } \dot{H}^{s_k}(\R).$
\item[$(ii)$] $ V(t)\phi_n \rightharpoonup V(t)\phi \textrm{ weakly in } L_x^{5k/4}L^{5k/2}_{t}.$
\end{itemize}
\end{lemma}

\begin{proof}[Proof of Lemma \ref{Vweak}]
Our argument follows that of Lemma 3.63 in \cite{MV1998}.
First suppose that $V(t) \phi_{n} \rightharpoonup V(t) \phi$ weakly in $L^{5k/4}_{x} L^{5k/2}_{t}$, let $\chi_{[-1,1]}$ be the characteristic function of the unit interval, and let $\psi \in \mathcal{S}(\R)$.
Set $$F(x,t) = \chi_{[-1,1]}(t) V(t) D^{2s_{k}}_{x} \psi(x).$$ Then an easy calculation reveals
\begin{equation*}
\int_{-\infty}^{\infty} \int_{-\infty}^{\infty} \Big ( V(t) \phi_{n}(x) \Big ) F(x,t) dx dt = 2 \int_{-\infty}^{\infty} \Big ( D^{s_{k}}_{x}\phi_{n}(x) \Big ) \Big ( D^{s_{k}}_{x} \psi(x) \Big ) dx.
\end{equation*}
It follows that $D^{s_{k}}_{x} \phi_{n} \rightharpoonup D^{s_{k}}_{x} \phi$ weakly in $L^{2}(\R)$, as required.

To prove the converse, we suppose that $\psi \in L^{5k/(5k-4)}_{x} L^{5k/(5k-2)}_{t}$ and observe that
\begin{align*}
\int_{-\infty}^{\infty} \int_{-\infty}^{\infty} \Big ( V(t) \phi_{n}(x) \Big ) \psi(x,t) dx dt &=\!\! \int_{-\infty}^{\infty} \phi_{n}(x) \left ( \int_{\R} V(t) \psi(x,t) dt \right ) dx\\
&=\!\! \int_{-\infty}^{\infty} \Big ( D^{s_{k}}_{x} \phi_{n}(x) \Big ) \left ( D^{-s_{k}}_{x} \int_{\R} V(t) \psi(x,t) dt \right ) dx.
\end{align*}
The result follows once we show that
\begin{equation*}
\left \| D^{-s_{k}}_{x} \int_{-\infty}^{\infty} V(t) \psi(\cdot,t) dt \right \|_{L^{2}} \leq c \| \psi \|_{L^{5k/(5k-4)}_{x} L^{5k/(5k-2)}_{t}}.
\end{equation*}
Indeed, by duality and a $TT^{\ast}$ argument we deduce from Lemma \ref{lemma1} $(i)$
\begin{equation}
\label{idualestimate}
\left \| \int_{-\infty}^{\infty} D^{-2s_{k}}_{x} V(t-s) \psi(\cdot,t) ds \right \|_{L^{5k/4}_{x} L^{5k/2}_{t}} \leq c \| \psi \|_{L^{5k/(5k-4)}_{x} L^{5k/(5k-2)}_{t}}.
\end{equation}
Finally, the desired estimate then follows from \eqref{idualestimate} by way of duality and a $TT^{\ast}$ argument.
\end{proof}

Using the previous lemma and the fact that $t^{j_0}=\lim_{n\rightarrow \infty} -\frac{t^{j_0}_n}{(h_n^{j_0})^3}=0$ we finally conclude
$$
(h_n^{j_0})^{2/k}u(h_n^{j_0}\cdot+x_n^{j_0},t_n)\rightharpoonup \psi^{j_0}\left(\cdot\right)\textrm{ weakly in } \dot{H}^{s_k}(\R).
$$
where $\|\psi^{j_0}\|_{\dot{H}^{s_k}}\geq \delta_0$ and $\{h_n^{j_0}\}_{n\in \N}$ satisfies \eqref{Blowuptime}.
Therefore, for every $R>0$ it follows that
$$
\lim_{n\rightarrow \infty} (h_n^{j_0})^{2s_k+4/k}\int_{|x|\leq R}|D^{s_k}_xu(h_n^{j_0}x+x_n^{j_0},t_n)|^2dx\geq \int_{|x|\leq R}|D^{s_k}_x\psi^{j_0}(x)|^2dx.
$$
The change of variables $y=h_n^{j_0}x+x_n^{j_0}$ yields
$$
\lim_{n\rightarrow \infty} \left(\sup_{z\in \R} \int_{|x-z|\leq h_n^{j_0}R}|D^{s_k}_xu(x,t_n)|^2dx\right)\geq \int_{|x|\leq R}|D^{s_k}_x\psi^{j_0}(x)|^2dx.
$$
Taking $\lambda(t)$ such that $\lambda(t)^{-1} (T^{\ast}-t)^{1/3}\rightarrow 0$ as $t\rightarrow T^{\ast}$, from the relation \eqref{Blowuptime} we deduce $\lambda(t_n)^{-1} h_n^{j_0} \rightarrow 0$ as $n\rightarrow \infty$, which implies 
$$
\lim_{n\rightarrow \infty} \left(\sup_{z\in \R} \int_{|x-z|\leq \lambda(t_n)}|D^{s_k}_xu(x,t_n)|^2dx\right)\geq \int_{|x|\leq R}|D^{s_k}_x\psi^{j_0}(x)|^2dx
$$
for every $R>0$. 
Since $\|\psi^{j_0}\|_{\dot{H}^{s_k}}\geq \delta_0$, we deduce
$$
\liminf_{t\rightarrow T^{\ast}} \left(\sup_{z\in \R} \int_{|x-z|\leq \lambda(t)}|D^{s_k}_xu(x,t_n)|^2dx\right)\geq \delta^2_0.
$$
Finally, using a continuity argument, we can find a function $x(t)\in \R$ such that \eqref{Concentration2} holds.
\end{proof}

%

\bibliographystyle{mrl}

\end{document}